\documentclass[12pt]{amsart}

\usepackage{amsmath, amssymb, amscd, newlfont}

\newcommand{\bbA}{{\mathbb{A}}}

\newcommand{\bbC}{{\mathbb{C}}}

\newcommand{\bbZ}{{\mathbb{Z}}}

\newcommand{\supp}{{\mathrm{supp}}}

\newcommand{\wit}{\widetilde}

\newcommand{\trace}{{\mathbf{tr}}}

\usepackage[mathscr]{eucal}

\numberwithin{equation}{section}

\newtheorem{Prop}[equation]{Proposition}
\newtheorem{Lem}[equation]{Lemma}
\newtheorem{Def}[equation]{Definition}
\newtheorem{Thm}[equation] {Theorem}
\newtheorem{Cor}[equation]{Corollary}
\newtheorem{Rem}[equation]{Remark}

\title
   [ Smooth Cuspidal Automorphic Forms]
   { Smooth Cuspidal Automorphic Forms and Integrable discrete Series}
\author{ Goran Mui\' c}

\address{ Department of Mathematics,
University of Zagreb,
Bijeni\v cka 30, 10000 Zagreb,
Croatia}
\email{gmuic@math.hr}

\subjclass{11E70, 22E50}
\keywords{}
\thanks{The  author acknowledges Croatian Science Foundation grant no. 9364.}

\begin{document}
\maketitle

\begin{abstract}
  In this paper we construct smooth cuspidal automorphic forms  related to integrable discrete series of a connected semisimple
  Lie group with finite center for classical and adelic situation as an application of the theory of Schwartz spaces for automorphic
  forms developed  by Casselman. In the classical situation, smooth cuspidal automorphic forms are constructed via an explicit
  continuous map from the Frech\' et space of  smooth vectors of a Banach realization inside $L^1(G)$ of an integrable discrete series
  into the space of smooth vectors of a strong topological dual of an appropriate Schwartz space. 
\end{abstract}

\section{Introduction}
The usual definition of an automorphic form includes the assumption that the function is $K$--finite on the Archimedean part.
Smooth automorphic forms are used extensively in the theory of
automorphic forms (see for example \cite{cogdell}, \cite{lapid}, \cite{ms1}, \cite{ms2}, \cite{ms3}).
Although  important, there is no explicit general construction of
smooth cuspidal forms which are not $K$--finite.

In our papers \cite{MuicMathAnn}, \cite{MuicJNT}, and \cite{MuicJFA} we have
studied the classical construction \cite{bb}
of $K$--finite cuspidal automorphic forms via Poincar\' e series from $K$--finite  matrix coefficients of integrable discrete
series. In this paper we complete these investigation by considering the  smooth case using  results of
Casselman \cite{casselman-1}.  We explain our results and content of the paper by sections. 

In this introduction, $G$ is a group of $\mathbb R$--points of a
semisimple algebraic group $\mathcal G$ defined over $\mathbb Q$. We assume that $G$ is not
compact and it is connected. Then, $G$ is a connected semisimple Lie group with finite center. In some sections of the paper
we just assume the latter. Important groups such as $Sp_{2n}(\mathbb R)$ or its double cover belong to this class.
In any case, we let $K$ be a choice of a maximal compact subgroup of $G$, and
 $\cal Z(\mathfrak g_{\mathbb C})$ be the center of the universal enveloping algebra of the complexified Lie algebra of $G$.

We let $\Gamma\subset G$ be a congruence subgroup with respect to the arithmetic structure
given by the fact that $\mathcal G$ defined over $\mathbb Q$ (see \cite{BJ}).
Then,  $\Gamma$ is a discrete subgroup of $G$ and it has a finite covolume.

In Section \ref{rmg}, we recall the notion of the norm on $G$ and state some properties of the norm.
This is essential to all other investigations in the paper.
 In Section \ref{paf}, we recall the
 definition of $K$--finite and smooth automorphic and cuspidal forms. We recall some basic properties of automorphic forms, and,
 in particular, the result that claims that 
a $\cal Z(\mathfrak g_{\mathbb C})$--finite and $K$--finite function in $L^p(\Gamma\backslash G)$ for some $p\ge 1$
is an automorphic form (see Lemma \ref{paf-1}). In Section \ref{src}, we give simple and natural proof of this result as an
application of results of Casselman \cite{casselman-1} (see Proposition \ref{src-6}). Besides, in Section \ref{src}, we prove Lemma
\ref{src-5} in which we prepare results of Casselman for application to the construction of smooth automorphic forms.
Aforementioned result about  automorphic forms is also a  consequence
of that lemma.

In Section \ref{ids},  we give a complete description 
of irreducible closed admissible subrepresentations of $L^1(G)$ under the right translations. It is quite likely that this is
well--known but we could not find a convenient reference. Such irreducible representations are integrable discrete series. A complete
description of them can be found in \cite{milicic}. The main results of Section \ref{ids} are contained in Lemma \ref{ids-2} and
Theorem \ref{ids-204}. In Lemma \ref{ids-0000} we prove that all smooth matrix coefficients of a discrete series representation belong to
$L^p(G)$, for some $p\in [1, 2[$, whenever there exists a non--zero $K$--finite matrix coefficient which satisfies the same. The proof is based
    on Casselman--Wallach theory of globalization of $(\mathfrak g, K)$---modules \cite{W2}, \cite{casselman}.  It is the key ingredient in the
    description of smooth vectors of realizations of integrable discrete series in $L^1(G)$
(see Lemma \ref{ids-2} (vii)).

In Section \ref{aaf}, after all aforementioned preparations, we prove the main result of the present paper
(see Theorem \ref{aaf-3}). There, we just assume that  $G$ is a connected semisimple Lie group with finite center and $\Gamma$ is
any discrete subgroup.  We assume that $G$ admits discrete series, and let $(\pi, \mathcal H)$ be an integrable discrete series
of $G$.  We fix a closed irreducible subrepresentation $\mathcal B_{h'}$ of $L^1(G)$ infinitesimally equivalent to $(\pi, \mathcal H)$
which is attached to a $K$--finite vector $h'\in \mathcal H$ via formation of certain matrix coefficients (see
Lemma \ref{ids-2}). In  Theorem \ref{aaf-3}, we look at the usual formation of Poincar\' e series $P_\Gamma$
attached to $\Gamma$ (see the first paragraph of Section \ref{aaf}), in two ways.

Firstly, we prove that
the map $\varphi \longmapsto P_\Gamma(\varphi)$ is a continuous  $G$--equivariant map from  the Banach representation
$\mathcal B_{h'}$ into the unitary representation  $L^2(\Gamma \backslash G)$. When $\Gamma$ is a congruence subgroup,
as an immediate consequence of asymptotic results of Wallach \cite{W0}, the image is contained in the cuspidal
subspace $L^2_{cusp}(\Gamma\backslash G)$. 

Secondly, we may consider $P_\Gamma$ as the continuous map $P_\Gamma: \ \mathcal B_{h'}\longrightarrow 
\mathcal S\left(\Gamma\backslash G\right)'$ where $\mathcal S\left(\Gamma\backslash G\right)'$ is the strong topological dual
of the Schwartz space $\mathcal S\left(\Gamma\backslash G\right)$. In \cite{casselman-1} (see also
Lemma \ref{src-4} in this paper),
it was proved that the Garding space of $\mathcal S\left(\Gamma\backslash G\right)'$ is the space of functions of uniform moderate
growth  $\cal A_{umg}(\Gamma\backslash G)$. In Theorem \ref{aaf-3} (ii), we prove that the space of smooth vectors
$\mathcal B_{h'}^\infty$ gets mapped into the subspace of  $\cal Z(\mathfrak g_{\mathbb C})$--finite vectors in
   $\cal A_{umg}(\Gamma\backslash G)$.  We remind the reader that when $\Gamma$ is a congruence subgroup, 
$\cal Z(\mathfrak g_{\mathbb C})$--finite vectors in $\cal A_{umg}(\Gamma\backslash G)$ are smooth automorphic forms by definition.
In this way, we achieve a construction of smooth cuspidal automorphic forms (see Theorem \ref{1aaf-4} for details). We remark that we have a
canonical isomorphism of Frech\' et representations $\mathcal B_{h'}^\infty\simeq \mathcal H^\infty$ (see Lemma \ref{ids-2} (vii)).

Thirdly, the map $P_\Gamma$ could be identically zero, but in Theorem \ref{aaf-3} (iv) we give a sufficient condition
that the map is not zero. It is based on our usual non--vanishing criterion (\cite{MuicMathAnn}, Theorem  4-1).

In Section \ref{1aaf}, we give applications of Theorem \ref{aaf-3}. In Theorem \ref{1aaf-4} we construct smooth cuspidal
automorphic forms, and we study when the map $P_\Gamma$ is not zero for principal congruence subgroups.  The corollary of these
investigations is  the result for smooth adelic cuspidal automorphic forms which we state and prove in Corollary \ref{1aaf-5}.

The first draft of the paper was written while the author visited the Hong Kong University of Science and Technology in May of 2016.
The author would like to thank A. Moy and the  Hong Kong University of Science and Technology for their hospitality.

\section{Norms on The Group}\label{rmg}

In this section we assume that  $G$ is a connected semisimple Lie group
with finite center, and  recall the notion of the norm on $G$.
It is essential for all what follows. 

\medskip

We fix a minimal parabolic subgroup $P=MAN$ of $G$ in
the usual way (see \cite{W1}, Section 2).
We have the Iwasawa decomposition $G=NAK$.

\medskip

We recall the notion of a  norm on the group following \cite{W1}, 2.A.2.
A norm $|| \ ||$ is a function $G\longrightarrow [1, \infty[$ satisfying
    the following properties:
    \begin{itemize}
    \item[(1)] $||x^{-1}||=||x||$, for all $x\in G$;
    \item[(2)]$||x\cdot y||\le ||x||\cdot || y||$, for all $x, y\in G$;
    \item[(3)] the sets $\left\{x\in G; \ \ ||x||\le r \right\}$ are compact  for all $r\ge 1$;
    \item[(4)] $||k_1\exp{(tX)} k_2||=||\exp{(X)}||^t$, for all $k_1, k_2\in K,
      X\in \mathfrak p, \ \ t\ge 0$.
    \end{itemize}
    Any two norms $||\ ||_i$, $i=1,2$, are equivalent: there exist $C, r>0$ such that
    $||x||_1\le C ||x||^r_2$, for all $x\in G$.

    Let $\Phi(\mathfrak g, \mathfrak a)$ be the set of all roots of $ \mathfrak a$ in
    $ \mathfrak g$. Let $\Phi^+(\mathfrak g, \mathfrak a)$ be the set of positive roots
    with respect to $\mathfrak n=Lie(N)$. Set

$$
\rho(H)= \frac12 \trace{\left(\text{ad}(H)|_{\mathfrak n}\right)}, \ \ H\in \mathfrak a.
$$
We set 
$$
m(\alpha)=\dim \ \mathfrak g_\alpha, \ \ \alpha\in \Phi^+(\mathfrak g, \mathfrak a).
$$
For $\mu\in \mathfrak a^\star$, we let 
$$
a^\mu = \exp{(\mu(H))}, \ \ a=\exp{(H)}.
$$
We define $A^+$ to be the set of all $a\in A$ such that $a^\alpha>1$ for all 
$\alpha\in \Phi^+(\mathfrak g, \mathfrak a)$.
Finally, we let
$$
D(a)=\prod_{\alpha\in \Phi^+(\mathfrak g, \mathfrak a)} \sinh{(\alpha(H))}^{m(\alpha)}, \ \  a=\exp{(H)}.
$$
Then, we may define a Haar measure on $G$ by the following formula:

$$
\int_G f(g) dg = \int_K \int_{A^+}\int_K D(a) f(k_1ak_2)dk_1 da dk_2, \ \  f\in C_c^\infty(G).
$$

Let $\left\{\alpha_1, \ldots, \alpha_r\right\}$ be the set of simple roots in
$ \Phi^+(\mathfrak g, \mathfrak a)$. Since $G$ is semisimple,
we have that this set spans $\mathfrak a^\star$. We define the dual basis
$\left\{H_1, \ldots, H_r\right\}$ of $\mathfrak a$ 
in the standard way: $\alpha_i(H_j)=\delta_{ij}$. By (\cite{W1}, Lemma 2.A.2.3),
there exists, $\mu, \eta\in  \mathfrak a^\star$
such that $\mu(H_i), \eta(H_j)>0$, for all $j$, and constants $C, D>0$ such that
$$
C a^\mu \le ||a||\le D a^\eta, \ \ a\in Cl(A^+).
$$

We remark that $\rho(H_j)>0$ for all $j$. So, we can find $c, d>0$ such that

$$
a^{c\rho} \le a^\mu  \le a^\eta \le a^{d\rho}, \ \ a\in Cl(A^+).
$$

We record this in the  next lemma:

\begin{Lem}\label{rmg-1} There exists real constants $c, C, d, D>0$ such that 
  $$
C a^{c\rho}  \le ||a||\le D a^{d\rho}, \ \ a\in Cl(A^+).
$$
\end{Lem}

A consequence of above integration formula and Lemma \ref{rmg-1}
is the following lemma (see \cite{W1}, Lemma 2.A.2.4):

\begin{Lem}\label{rmg-2} Maintaining above assumptions,  we have  $\int_G ||g||^{-m} dg<\infty$  for $m> \max_{1\le i \le r} \frac{1}{c \rho(H_i)}$.
\end{Lem}
\begin{proof} Let $m>0$. Then, by above properties of the norm and the integration formula
  $$\int_G ||g||^{-m}dg= \int_{A^+} D(a)||a||^{-m} da=\int_{\substack{H\in \mathfrak a\\
      \alpha_1(H)>0,\ldots
      \alpha_r(H)>0  }}  D(\exp{H})\ ||\exp{H}||^{-m} dH,
  $$
  where $dH$ is any Euclidean measure on $\mathfrak a$. We fix a basis $H_1, \ldots, H_r$ described above. In this basis,
  the right--hand side  becomes
 $$
 \int_{\substack{t_1>0,\ldots t_r>0  }}  D\left(\exp{\left(\sum_{i=1}^r t_iH_i\right)}\right)\ 
 \left|\left|\exp{\left(\sum_{i=1}^r t_iH_i\right)}\right|\right|^{-m} dt_1\cdots dt_r.
 $$
 Obviously, the definition of $D$ implies
$$
 D\left(\exp{\left(\sum_{i=1}^r t_iH_i\right)}\right)\le  \exp{\left(\sum_{i=1}^r \rho(H_i)t_i\right)} ,
 $$
 for $t_i>0$, $i=1, \ldots, r$. By Lemma \ref{rmg-1}, there exists real constants $c, C$ such that 
  $$
C \exp{\left(\sum_{i=1}^r c\rho(H_i)t_i\right)}  \le  \left|\left|\exp{\left(\sum_{i=1}^r t_iH_i\right)}\right|\right| ,$$
for $t_i>0$, $i=1, \ldots, r$. Hence, the integral is 
$$
\le  C^{-m}\int_{\substack{t_1>0,\ldots t_r>0  }}  \exp{\left(\sum_{i=1}^r  \left(1-m c \rho(H_i)\right)t_i\right)} 
dt_1\cdots dt_r.
$$
By elementary calculus, the integral is finite for
$$
m> \max_{1\le i \le r} \frac{1}{c \rho(H_i)}.
$$
\end{proof}

\section{Preliminaries on Automorphic Forms}\label{paf}

In this section we assume that
 $G$ is a group of $\mathbb R$--points of a
semisimple algebraic group $\mathcal G$ defined over $\mathbb Q$. Assume that $G$ is not
compact and connected.
Let $\Gamma\subset G$ be congruence subgroup with respect to the arithmetic structure
given by the
fact that $\mathcal G$ defined over $\mathbb Q$ (see \cite{BJ}).
Then,  $\Gamma$ is a discrete subgroup
of $G$ and it has a finite covolume.

\medskip
An automorphic form (or a $K$--finite automorphic form; see \cite{cogdell}) for $\Gamma$
is a function $f\in C^\infty(G)$ satisfying the following
three conditions (\cite{W3} or \cite{BJ}):

\begin{itemize}
\item[(A-1)] $f$ is  $\cal Z(\mathfrak g_{\mathbb C})$--finite and $K$--finite on the right;
\item[(A-2)] $f$ is left--invariant under $\Gamma$ i.e., $f(\gamma x)=f(x)$ for all 
  $\gamma\in \Gamma$, $x\in G$;
\item[(A-3)] there exists $r\in\mathbb R$, $r>0$
  such that for each $u\in \mathcal U(\mathfrak g_{\mathbb C})$ there exists a constant
  $C_u>0$ such that $\left|u.f(x)\right|\le C_u \cdot ||x||^r$, for all $x\in G$.
\end{itemize}
A smooth automorphic form (see \cite{casselman-1}, \cite{cogdell}) for $\Gamma$
is a function $f\in C^\infty(G)$ satisfying (A1)--(A3) except possibly $K$--finiteness.
We discuss smooth automorphic forms in more detail the next section.

\medskip
We write $\mathcal A(\Gamma\backslash G)$ (resp., $\mathcal A^\infty(\Gamma\backslash G)$)
for the vector space of all automorphic forms (resp., smooth automorphic forms). Obviously, 
$\mathcal A(\Gamma\backslash G) \subset \mathcal A^\infty(\Gamma\backslash G)$. 
It is easy to see that $\mathcal A(\Gamma\backslash G)$ is a $(\mathfrak g, K)$--module (using \cite{hc},
Theorem 1 and an argument simiar to the one used in the proof of Lemma \ref{paf-10}),
and since $G$ is connected, the space $\mathcal A^\infty(\Gamma\backslash G)$
is $G$--invariant . An automorphic form
$f\in \mathcal A^\infty(\Gamma\backslash G)$  is a $\Gamma$--cuspidal automorphic form if for every proper
$\mathbb Q$--proper parabolic subgroup $\mathcal P\subset \mathcal G$ we have
$$
\int_{U\cap \Gamma \backslash U} f(ux)dx=0, \ \ x\in G,
$$
where $U$ is the group of $\mathbb R$--points of the unipotent radical of  $\mathcal P$.
We remark that the quotient $U\cap \Gamma \backslash U$ is compact. We use normalized $U$--invariant measure on
$U\cap \Gamma \backslash U$. The space of all $\Gamma$--cuspidal automorphic forms (resp.,
$\Gamma$--cuspidal smooth automorphic forms) for $\Gamma$
is denoted by $\mathcal A_{cusp}(\Gamma\backslash G)$ (resp., $\mathcal A^\infty_{cusp}(\Gamma\backslash G)$).
The space $\mathcal A_{cusp}(\Gamma\backslash G)$ is a   $(\mathfrak g, K)$--submodule of $\mathcal A(\Gamma\backslash G)$.
The space $\mathcal A^\infty_{cusp}(\Gamma\backslash G)$ is $G$--invariant.

\vskip .2in 
Following Casselman \cite{casselman-1}, we define
$$
||g||_{\Gamma\setminus G}=\inf_{\gamma\in \Gamma} ||\gamma g||, \ \ g \in G.
$$
It is obvious that $||\cdot ||_{\Gamma\setminus G}$ is $\Gamma$--invariant on the right, and that
$||g||_{\Gamma\setminus G}\le ||g||$ for all $g\in G$. The condition (A-3) is equivalent to
\begin{itemize}
\item[(A-3')] there exists $r\in\mathbb R$, $r>0$
  such that for each $u\in \mathcal U(\mathfrak g_{\mathbb C})$ there exists a constant
  $C_u>0$ such that $\left|u.f(x)\right|\le C_u \cdot ||x||^r_{\Gamma\setminus G}$, for all $x\in G$.
\end{itemize}

\vskip .2in
We recall the following standard result:

\begin{Lem}\label{paf-1} Under above assumptions, we have the following:
  \begin{itemize}
         \item[(a)] If  $f\in C^\infty(G)$ satisfies (A-1), (A-2), and
           there exists  $p\ge 1$ such that $f\in L^p(\Gamma\backslash G)$, then f satisfies (A-3), and it is
           therefore an
           automorphic form. We speak about $p$--integrable automorphic form, for $p=1$ (resp., $p=2$) we speak
           about integrable
        (resp., square--integrable) automorphic form.
      \item[(b)] Let $p\ge 1$. Every $p$--integrable automorphic form is integrable.
      \item[(c)] Bounded integrable automorphic form is square--integrable.
      \item[(d)] If $f$ is square integrable automorphic form, then the minimal $G$--invariant closed subspace of
        $L^2(\Gamma\backslash G)$ is a direct is of finitely many irreducible unitary representations.
       \item[(e)] Every $\Gamma$--cuspidal automorphic form is square--integrable.
    \end{itemize}
  \end{Lem}
\begin{proof} For the claims (a) and (e) we refer to \cite{BJ} and reference there.
  Since the volume of $\Gamma \backslash G$ is finite, the claim (b) follows from H\" older inequality (as in
  \cite{MuicMathAnn}, Section 3). The claim (c) is obvious. The claim (d) follows from (\cite{W1}, Corollary
  3.4.7 and Theorem  4.2.1).
\end{proof}

\medskip
Proposition \ref{src-6} gives simple proof of Lemma \ref{paf-1} (a) using some
results of Casselman \cite{casselman-1}.

\medskip 
We include the proof of the following standard result since it will be useful in
clarification  of various issues in the next section. 
\medskip

\begin{Lem}\label{paf-10} If $f\in C^\infty(G)$ satisfies (A-1), (A-2) and
  there exists constants $r>0$, $C>0$ such that $\left|f(x)\right|\le C \cdot ||x||^r$, for all
  $x\in G$, then (A-3)
      also holds. 
\end{Lem}
\begin{proof} By the assumption (A-1) and
  a theorem of Harish--Chandra (see \cite{hc}, Theorem 1) there exists $\alpha\in C_c^\infty(G)$ such that
  $f=f\star \alpha$.  This implies that for $u\in \mathcal U(\mathfrak g_{\mathbb C})$ we have
  $$
  uf(x)=\int_G f(y) u\alpha(y^{-1}x)dy=\int_G f(xy^{-1}) u\alpha(y) dy.
  $$
  Hence, by using the assumption and properties (1) and (2) of the norm (see Section \ref{rmg})
  \begin{align*}
    |uf(x)|&\le \int_G \left|f(xy^{-1})\right| \cdot \left| u\alpha(y)\right| dy
           \le C \int_G ||xy^{-1}||^r \cdot \left| u\alpha(y)\right| dy\\
    & \le C ||x||^r \int_G  ||y||^r  \cdot \left| u\alpha(y)\right| dy
  \end{align*}
  which proves the claim.
  \end{proof}

\section{Some Results of Casselman}\label{src}

In this section we  assume that $G$ is a semisimple connected Lie group with finite center.
We assume that $\Gamma$ is a discrete subgroup of $G$. For example, $\Gamma$ could be a congruence subgroup or just
a trivial group. Main result of this section is observation in
Lemma \ref{src-5} used in Section \ref{aaf} in the construction of smooth automorphic forms. 

\medskip
We recall the definition of the Schwartz space
$\mathcal S\left(\Gamma\backslash G\right)$ defined by Casselman 
(\cite{casselman-1}, page 292). It consists of all functions $f\in C^\infty(G)$ satisfying the following
conditions:
\medskip

\begin{itemize}
\item[(CS-1)] $f$ is left--invariant under $\Gamma$ i.e., $f(\gamma x)=f(x)$ for all
  $\gamma\in \Gamma$, $x\in G$;
\item[(CS-2)] $||f||_{u, -n}<\infty$ for all $u\in \mathcal U(\mathfrak g_{\mathbb C})$, and all natural
  numbers $n\ge 1$.
\end{itemize}

\medskip
In above definition, for  $u\in \mathcal U(\mathfrak g_{\mathbb C})$, and a real number $s$, we let
$$
||f||_{u, s}\overset{def}{=} \sup_{x\in G}  ||x||_{\Gamma\setminus G}^{-s} \left|u.f(x)\right|.
$$
Since $||x||_{\Gamma\setminus G} \ge 1$, we have
$$
||f||_{u, s'}\le ||f||_{u, s},
$$
for $s'>s$.

\vskip .2in 
We recall the following result (see \cite{casselman-1}, 1.8 Proposition):

\begin{Prop}\label{src-1} Using above notation, we have the following:
  \begin{itemize}
  \item[(i)]  The Schwartz space $\mathcal S\left(\Gamma\backslash G\right)$ is a Fr\' echet space
    under the seminorms:
    $||\ ||_{u, -n}$, $u\in \mathcal U(\mathfrak g_{\mathbb C})$, $n\in\mathbb Z_{\ge 1}$.
  \item[(ii)] The right regular representation of $G$ on $\mathcal S\left(\Gamma\backslash G\right)$ is a smooth  Fr\' echet
    representation of moderate growth. 
    \end{itemize}
\end{Prop}

\medskip
We recall the definition of representation of moderate growth. Let $(\pi, V)$ be a continuous representation
on the Fr\' echet space   $V$. We say that $(\pi, V)$ is of moderate growth if its is smooth, and 
if for any semi-norm $\rho$ there exists an integer
$n$, a constant $C>0$, and another semi-norm $\nu$ such that 
$$
||\pi(g)v||_\rho\le C ||g||^n ||v||_\nu, \ \ g\in G, \ v\in V.
$$
We recall that the semi-norms on a locally convex vector space (for example, a Frech\' et space) $V$ are
constructed via Minkowski functionals.

\vskip .2in
The following definition is from (\cite{casselman-1}, page 295).

\begin{Def}\label{src-2} The space $\mathcal S\left(\Gamma\backslash G\right)'$ of tempered distributions or
  distributions of moderate growth  on $\Gamma \backslash G$ is the
  strong topological dual of $\mathcal S\left(\Gamma\backslash G\right)$. 
  \end{Def}

\medskip

For convenience of the reader,  we recall the definition of a strong topological dual in our particular case.
By general theory, the subset  $B\subset \mathcal S\left(\Gamma\backslash G\right)$ is bounded if for every neighborhood $V$ of
$0$ there exists $s>0$ such that $B\subset tV$, for $t>s$. This definition is not very practical to use.
Again from the general theory (and easy to see directly), $B\subset \mathcal S\left(\Gamma\backslash G\right)$
is bounded if and only if it is bounded in every semi-norm defining topology on $\mathcal S\left(\Gamma\backslash G\right)$
i.e., 
$$
\sup_{f\in B} ||f||_{u, -n} <\infty, \ \ u\in \mathcal U(\mathfrak g_{\mathbb C}), \ \ n\in  \mathbb Z_{\ge 1}.
$$

\medskip
The strong topological dual
$\mathcal S\left(\Gamma\backslash G\right)'$  of $\mathcal S\left(\Gamma\backslash G\right)$ is
the space of continuous functionals on $X$ equipped with strong topology i.e. topology of uniform convergence on bounded sets
in $\mathcal S\left(\Gamma\backslash G\right)$ i.e. topology given by semi--norms
$$
||\alpha||_B=\sup_{f\in B} \ \left|\alpha(f)\right|, \ \ \text{where $B$ ranges over bounded sets of
  $\mathcal S\left(\Gamma\backslash G\right)$}.
$$

\vskip.2in

By general theory of topological vector spaces, the space $\mathcal S\left(\Gamma\backslash G\right)'$ is complete
locally convex  (defined by above semi-norms) vector space. The natural action of $G$ on
$\mathcal S\left(\Gamma\backslash G\right)'$ is continuous. The usual representation--theoretic argument are valid
there (\cite{hc}, Section 2). 

\vskip .2in
The following lemma can be used to deal with the limits in $\mathcal S\left(\Gamma\backslash G\right)'$
but of course it is
not sufficient to deal with the topology on $\mathcal S\left(\Gamma\backslash G\right)'$. The proof is left as
an exercise to the reader.

\begin{Lem}\label{src-1000} Let $\alpha_n$, $n\ge 1$, be a sequence in $\mathcal S\left(\Gamma\backslash G\right)'$ and let
  $\alpha \in \mathcal S\left(\Gamma\backslash G\right)'$. Then, $\alpha_n\longrightarrow \alpha$ if for sufficiently large
  numbers
  $M>0$ and  $m\in \mathbb Z_{\ge 1}$ we have
  $$
  \lim_{n} \ \sup_{||f||_{1, -m}\le M} \ \left|\alpha_n(f)-\alpha(f)\right|\longrightarrow 0.
  $$
  \end{Lem}

\vskip .2in
Following Casselman, we consider the two spaces of functions: the functions of moderate growth
$\mathcal A_{mg}(\Gamma \backslash G)$, and
the functions of uniform moderate growth $\mathcal A_{umg}(\Gamma \backslash G)$. The space
$\mathcal A_{mg}(\Gamma \backslash G)$ consists of the functions $f\in C^\infty(G)$ satisfying the following conditions:

\vskip .2in

\begin{itemize}
\item[(MG-1)] $f$ is left--invariant under $\Gamma$ i.e., $f(\gamma x)=f(x)$ for all
  $\gamma\in \Gamma$, $x\in G$;
\item[(MG-2)] for each $u\in \mathcal U(\mathfrak g_{\mathbb C})$ there exists a constant
  $C_u>0$, $r_u\in\mathbb R$, $r_u>0$  such that $\left|u.f(x)\right|\le C_u \cdot ||x||^{r_u}$, for all $x\in G$.
\end{itemize}

\vskip .2in 
The space $\mathcal A_{umg}(\Gamma \backslash G)$ consists
of the functions $f\in C^\infty(G)$ satisfying the following conditions:

\vskip .2in 
\begin{itemize}
\item[(UMG-1)] $f$ is left--invariant under $\Gamma$ i.e., $f(\gamma x)=f(x)$ for all
  $\gamma\in \Gamma$, $x\in G$;
\item[(UMG-2)] there exists $r\in\mathbb R$, $r>0$
  such that for each $u\in \mathcal U(\mathfrak g_{\mathbb C})$ there exists a constant
  $C_u>0$ such that $\left|u.f(x)\right|\le C_u \cdot ||x||^r$, for all $x\in G$.
\end{itemize}
We note that in the second definition $r$ is independent of $u\in \mathcal U(\mathfrak g_{\mathbb C})$.

\medskip
\begin{Lem}\label{src-3} We maintain the assumptions of the first paragraph of Section \ref{paf}.
  Then, the spaces of functions which are $\cal Z(\mathfrak g_{\mathbb C})$--finite and $K$--finite on the right in
  $\mathcal A_{mg}(\Gamma \backslash G)$, and in $\mathcal A_{umg}(\Gamma \backslash G)$ coincide,
  and are equal to the  space $\mathcal A(\Gamma\backslash G)$ of automorphic forms for $\Gamma$. Next,
  the space of smooth automorphic forms $\mathcal A^\infty(\Gamma\backslash G)$ is a subspace of
  $\cal Z(\mathfrak g_{\mathbb C})$--finite functions in $\mathcal A_{umg}(\Gamma \backslash G)$.
  Furthermore, we have
  $$
  \mathcal A(\Gamma\backslash G)\subset \mathcal A^\infty(\Gamma\backslash G)\subset
  \mathcal A_{umg}(\Gamma \backslash G)\subset \mathcal A_{mg}(\Gamma \backslash G).
  $$
  \end{Lem}
\begin{proof} The first and second claims  follow from (A1)--(A3), and Lemma \ref{paf-10}. The third claim is obvious.
\end{proof}

\vskip .2in

\begin{Lem}\label{src-4} The Garding space in   $\mathcal S\left(\Gamma\backslash G\right)'$ is equal to the space
  $\mathcal A_{umg}(\Gamma \backslash G)$.
  \end{Lem}
\begin{proof} This (\cite{casselman-1}, Theorem 1.16).
\end{proof}

\vskip .2in
We remark that $\mathcal S\left(\Gamma\backslash G\right)'$ is not a Fr\' echet space  so \cite{dixmal} can not be applied
to prove that the space of smooth vectors is the same as the Garding space. Therefore, for example, in the settings of
Lemma \ref{src-3},
$\mathcal A^\infty(\Gamma\backslash G)$ is just subspace of the space of all $\cal Z(\mathfrak g_{\mathbb C})$--finite vectors in
$\mathcal S\left(\Gamma\backslash G\right)'$.

\vskip .2in
Regarding smooth vectors in $\mathcal S\left(\Gamma\backslash G\right)'$, the following remarkable
lemma will be used later:

\medskip 
\begin{Lem}\label{src-5} Assume that $f\in L^p(\Gamma\setminus G)$, for some $p\ge 1$, and $\alpha\in C_c^\infty(G)$. Then,
  $f\star \alpha$ is equal almost everywhere to a function in  $\mathcal A_{umg}(\Gamma \backslash G)$.
\end{Lem}
\begin{proof} It follows easily from above description of topology of $\cal S(\Gamma\backslash G)$ that
  $\varphi\longmapsto \int_{\Gamma\backslash G} F(x) \varphi(x) dx$ belongs to $\cal S(\Gamma\backslash G)'$
  for any  $F\in L^p(\Gamma\setminus G)$.

 The action of convolution algebra  $C_c^\infty(G)$ on $\cal S(\Gamma\backslash G)'$ is given by
  $$
r'(\alpha)\Lambda=  \int_G \alpha(x) \ r'(x)\Lambda \ dx, \ \ \Lambda \in \cal S(\Gamma\backslash G)', \alpha\in C_c^\infty(G),
  $$
where $r'(x)$ is the contragredient of the right translation $r(x)$ for $x\in G$. By definition,
we have the following:

$$
r'(\alpha)\Lambda(\varphi)=  \int_G \alpha(x) \ r'(x)\Lambda(\varphi) \ dx= \int_G \alpha(x)
\Lambda(r(x^{-1})\varphi) dx=
\Lambda\left(r(\alpha^\vee)\varphi\right).
$$

We remark that
$$
r(\alpha^\vee)\varphi (x) =\int_G \alpha(y^{-1})\varphi(xy) dy=\varphi\star \alpha(x).
$$

Now, let the functional $\Lambda$ be the functional given by the integration againts $f$ and $\Sigma$ be the functional
given by the integration against $f\star  \alpha$. Then we have the following computation:

\begin{align*}
  \Sigma(\varphi) &=\int_{\Gamma\backslash G} \left(f\star \alpha\right)(x) \varphi(x) dx\\
  &=\int_{\Gamma\backslash G} \left(\int_G f(xy^{-1}) \alpha(y) dy\right)\  \varphi(x) dx\\
  &=\int_{\Gamma\backslash G}f(x)  \left(\int_G \varphi(xy) \alpha(y) dy\right) dx\\
  &=\int_{\Gamma\backslash G}f(x) \varphi\star \alpha(x)  dx\\
  &=\Lambda\left(r(\alpha)\varphi\right)\\
  &=r'(\alpha)\Lambda\left(\varphi\right), \ \ \varphi \in \cal S(\Gamma\backslash G). 
  \end{align*}

Thus, $\Sigma$ belongs to the Garding space of $\cal S(\Gamma\backslash G)'$. Hence, by Lemma \ref{src-4}, there
exists
a $F\in \mathcal A_{mg}(\Gamma \backslash G)$ such that
$$
\int_{\Gamma\backslash G} \left(f\star \alpha\right)(x) \varphi(x) dx= \Sigma (\varphi)=
\int_{\Gamma\backslash G} F(x) \varphi(x) dx,
$$
for all $\varphi \in \cal S(\Gamma\backslash G)$. Since $C_c^\infty(\Gamma\backslash G)\subset
\mathcal S(\Gamma\backslash G)$,
the claim follows. 
\end{proof}

\vskip .2in
This result can be used to give a new and simple proof of important result stated in Lemma \ref{paf-1} (a).

\begin{Prop} \label{src-6}  If  $f\in C^\infty(G)$ satisfies (A-1), (A-2) of Section \ref{paf}, and
  there exists  $p\ge 1$ such that $f\in L^p(\Gamma\backslash G)$, then f satisfies (A-3) of Section \ref{paf},
  and it is therefore an automorphic form in $\mathcal A(\Gamma\backslash G)$.
\end{Prop}
\begin{proof} By (A-1), $f$ is  $\cal Z(\mathfrak g_{\mathbb C})$--finite and $K$--finite on the right.
  Them by  a theorem of Harish--Chandra (\cite{hc}, Theorem 1), there exists $\alpha\in C_c^\infty(G)$ such that
  $f=f\star \alpha$.  Hence, Lemma \ref{src-5} implies that $f\in \mathcal A_{umg}(\Gamma \backslash G)$. Now,
  (A-1) and   Lemma \ref{src-3} complete the proof.
  \end{proof} 

\section{On a Description of Certain Irreducible Subrepresentations in $L^1(G)$}\label{ids}

In this section we  assume that $G$ is a semisimple connected Lie group with finite center. We give a complete  description 
of irreducible closed admissible subrepresentations of $L^1(G)$ under the right translations.  

\medskip

Let $(\pi, \mathcal B)$ be a continuous representation of $G$  on the Banach space $\mathcal B$.
We denote by $\mathcal B^\infty$  the subspace of smooth vectors in 
$\mathcal H$. It is a complete Fr\' echet space under the family of semi--norms:

$$
||b||_u =||d\pi(u)b||, \ \ u\in \mathcal U(\mathfrak g_{\mathbb C}), \ \  b\in \mathcal B^\infty,
$$
where $||\ ||$  is the norm on $\mathcal B$. 
The natural representation $\pi^\infty$ of $\mathcal B^\infty$ is a smooth Fr\' echet representation
of moderate growth (\cite{W2}, Lemma 11.5.1).

 \vskip .2in
Let $\hat{K}$ be the set of
equivalence of irreducible representations of $K$. Let $\delta\in 
\hat{K}$, then we write $d(\delta)$ and $\xi_\delta$ the degree and
character of $\delta$, respectively. We fix the normalized Haar
measure $dk$ on $K$. For $h\in \mathcal H$, we let
$$
E_\delta(b)=\int_{K} d(\delta) \overline{\xi_\delta(k)}\ \pi(k)b\ dk, \ \ \mathcal B.
$$
It belongs to the $\delta$--isotypic component $\mathcal B(\delta)$ of
$\mathcal B$.   We have

\begin{align*}
&E_\delta E_\delta=E_\delta\\
&E_\delta E_\gamma=0 \ \ \text{if} \ \ \delta\not\simeq \gamma.\\
\end{align*}

\vskip .2in
We state the following lemma that we need later:

\begin{Lem} \label{ids-1}
  Let $b\in \mathcal B^\infty$. Then, we have the following:
  \begin{itemize}
  \item[(i)] There exists $b_1, \ldots, b_l\in \mathcal B^\infty$, and  $\alpha_1, \ldots, \alpha_l\in C_c^\infty(G)$ such that
    $$
    b=\sum_{i=1}^l \pi(\alpha_i)b_i.
    $$
  \item[(ii)]  We have the following expansion 
    $$
    b=\sum_{\delta\in \hat{K}} \ E_\delta(b)
    $$
    in above described Fr\' echet topology where the convergence is absolute:
  $$
  \sum_{\delta\in \hat{K}} \ \left|\left|d\pi(u)E_\delta(b)\right|\right|<\infty, \ \ \text{for all } \ u\in \mathcal U(\mathfrak g_{\mathbb C}).
  $$
  \end{itemize}
  \end{Lem}
  \begin{proof} (i) is of course a well--known result of Dixmier--Malliavin applied to the
    Banach representation $(\pi, \mathcal B)$. By the way, this implies that $(\pi^\infty, \mathcal B^\infty)$ is a smooth
    representation in its natural topology. One just needs to apply (\cite{hc}, Lemma 2). Now, having this remark, 
    (i) is just (\cite{hc}, Lemma 5) applied $(\pi^\infty, \mathcal B^\infty)$. 
  \end{proof}

\vskip .2in
We consider $L^1(G)$ as a Banach representation of $G$ under the right--translations $r$.
Let $\alpha\in C_c^\infty(G)$. It acts on $L^1(G)$ as follows:
$$
r(\alpha)f(x)=\int_G \alpha(y) f(xy)dy=\int_G f(xy^{-1})\alpha^\vee (y)dy= \int_Gf(y)\alpha^\vee (y^{-1}x) dy, \ \
f\in L^1(G).
$$
The function $r(\alpha)f$ belongs to $C^\infty(G)$ and for each $u\in \mathcal U(\mathfrak g_{\mathbb C})$ we have the following:
$$
u.r(\alpha)f(x)= \int_Gf(y)u.\alpha^\vee (y^{-1}x) dy.
$$
By definition, $r(\alpha)f$ belongs to a Garding space of the right--regular representation of $G$ on $L^1(G)$. Thus,
the vector $r(\alpha)f$ is smooth for that representation. Thus, 
$u\in \mathcal U(\mathfrak g_{\mathbb C})$ acts on $r(\alpha)f$. It is easy to see that the action is
described by above formula. For example, if $X\in \mathfrak g$ and $\alpha$ is real--valued, then we compute
\begin{align*}
&   \int_G \left|\frac{r(\alpha)f(x\exp{(tX)})- r(\alpha)f(x)}{t}- \left(f\star X.\alpha^\vee \right)(x)\right| dx\\
  &\le \int_G\int_G \left|f(y)\right| \left|\frac{\alpha^\vee (y^{-1}x\exp{(tX)})- \alpha^\vee (y^{-1}x)}{t}
  -X.\alpha^\vee (y^{-1}x)\right|dxdy\\
  &= \int_G \left|f(y)\right| dy \int_G \left|\frac{\alpha^\vee (x\exp{(tX)})- \alpha^\vee (x)}{t} -X.\alpha^\vee (x)
  \right|dx\\
   &=\int_G \left|f(y)\right| dy \int_G \left|X.\alpha^\vee (x\exp{(t_x X)})-X.\alpha^\vee (x)\right|dx,\\
\end{align*}
for some $t_x\in ]0, \ t[$ by the Mean value theorem. Letting $t\longrightarrow 0$, by the Dominated convergence theorem, we
    obtain the claim. Similar considerations hold for the left regular representation of $G$ on $L^1(G)$ denoted by $l$.    
  Let $\beta\in C_c^\infty(G)$. It acts on $L^1(G)$ as follows:
$$
l(\beta)f(x)=\int_G \beta(y) f(y^{-1}x)dy=\int_G \beta(xy^{-1})f(y)dy=\beta\star f(x).
$$
This function belongs to $C^\infty(G)$ and for each $u\in \mathcal U(\mathfrak g_{\mathbb C})$ we have the following:
 \begin{equation}\label{ids-2000a}
l(u)l(\beta)f = \left(l(u)\beta\right)\star f.
\end{equation}

Let $\delta\in  \hat{K}$. As before, if $dk$ is the normalized measure on $K$, then we let
\begin{align*}
  E^l_\delta(\cdot )&=\int_{K} d(\delta) \overline{\xi_\delta(k)}l(k) \ dk\\
   E^r_\delta(\cdot )&=\int_{K} d(\delta) \overline{\xi_\delta(k)}r(k) \ dk.
\end{align*}
For a finite set $S\subset \hat{K}$, we let
\begin{align*}
  E^l_S &=\sum_{\delta\in S} \ E^l_\delta\\
   E^r_S &=\sum_{\delta\in S} \ E^r_\delta\\
  \end{align*}

\vskip .2in
Let $(\pi, \mathcal H)$ be an irreducible unitary representation of $G$. 
We write $\langle \ , \ \rangle$ for the $G$--invariant inner product on $\mathcal H$.
It is a well--known result of Harish--Chandra that it is admissible.
 Let $\mathcal H_K$ be the space of
 $K$--finite vectors. It is a  $(\mathfrak g, K)$--module, and it is dense in
 $\mathcal H^\infty$ in the Fr\' echet topology.  The space  $\mathcal H^\infty$ is a smooth Fr\' echet representation
of moderate growth (\cite{W2}, Lemma 11.5.1).
 
\vskip .2in
A matrix coefficient of $\pi$ is a function on $G$ of the form $x\longmapsto
 \langle \pi(x)h, \ h'\rangle$, where
 $h, h'\in\mathcal H$. Obviously, $\varphi\neq 0$ if and only if $h, h'\neq 0$.
 The matrix coefficient is $K$--finite on
 the right (resp., on the left and on both sides)  if and only if $h\in \mathcal H_K$
 (resp.,  $h'\in \mathcal H_K$ and
 $h, h'\in \mathcal H_K$). The matix coefficient is smooth if $h, h'\in \mathcal H^\infty$

  \vskip .2in
    From now on we assume that $G$ admits discrete series. By the well--known classification of discrete series due to
    Harish--Chandra, this is the case if and only if $rank(G)=rank(K)$. A unitary representation $(\pi, \mathcal H)$
    is in discrete series if it has a non--zero matrix coefficient which belongs to $L^2(G)$. Due to results of
    Mili\v ci\' c \cite{milicic}, most of discrete series poses a non--zero $K$--finite matrix in  $L^1(G)$. The precise
    description of such representations in terms of Harish--Chandra parameters is contained in \cite{milicic}. We say that
    $(\pi, \mathcal H)$ is integrable if it is in discrete series and it has a non--zero $K$--finite matrix coefficient in 
    $L^1(G)$ (then all $K$--finite matrix coefficients are in $L^1(G)$).  In fact, it is an exercise to prove
    that if $(\pi, \mathcal H)$ has a non--zero $K$--finite matrix coefficient in $L^1(G)$, then this matrix coefficient is in
    $L^2(G)$, and consequently $(\pi, \mathcal H)$ is in the discrete series (see the argument in the proof of Lemma
    \ref{ids-200}).

    \vskip .2in
    The following lemma is important in our investigations below. It is a consequence of deep results of Casselman and Wallach on the
    globalization of $(\mathfrak g, K)$--modules (\cite{W2}, Chapter 11, or \cite{casselman}).

    \begin{Lem} \label{ids-0000} Assume that $(\pi, \mathcal H)$ is representation in the discrete series such that
      it poses a non--zero $K$--finite matrix coefficient which belongs to $L^p(G)$ for some $p\in [1, 2[$.
          Then, all smooth matrix coeffcients belong to
      $L^p(G)$ as well.
    \end{Lem}
    \begin{proof}  The space  $\mathcal H^\infty$ is a smooth Fr\' echet representation
      of moderate growth (\cite{W2}, Lemma 11.5.1). On that space the Fr\' echet algebra $\mathcal S(G)$ acts (see Section \ref{src}
      for the definition of the space $\mathcal S(G)$; the action is described in \cite{W2}, 11.8).  By (\cite{W2}, Theorem 11.8.2),
      $\mathcal H^\infty$ is irreducible as an algebraic module for $\mathcal S(G)$. In particular, for $h{''}\in  \mathcal H^\infty$, $h{''}\neq 0$,
      we have $\mathcal H^\infty=\mathcal S(G)h{''}$. Now, let us consider the matrix coefficient $c_{h, h'}$, where $h, h'\in \mathcal H^\infty$.
      Select $h{''}\in \mathcal H_K$ such that $h{''}\neq 0$. Then, since there exists $K$--finite matrix coefficients in $L^p(G)$, then all of them are
      in $L^p(G)$. In particular, we have $c_{h{''}, h{''}}\in L^p(G)$. Next, we select $f, f'\in \mathcal S(G)$ such that $h=\pi(f)h{''}$ and
      $h'=\pi(f')h{''}$.   Now, since $\pi$ is unitary, we have the following:
      \begin{align*}
        c_{h, h'}(x)&=\langle \pi(x)h, \ h'\rangle=\langle \pi(x)\pi(f)h{''}, \ h\pi(f')h{''}\rangle\\
        &=\int_G\int_G  f(g)\overline{f'(g')} \langle \pi((g')^{-1}xg)h{''}, \ h{''}\rangle dgdg'.
      \end{align*}
      Hence, using Remark \ref{convex}, we have 
      \begin{align*}
      \int_G \left|c_{h, h'}(x)\right|^p dx & \le \int_G\int_G \int_G
      \left|f(g)\right| \cdot \left|\overline{f'(g')}\right| \left|\langle \pi((g')^{-1}xg)h{''}, \ h{''}\rangle\right|^pdx dgdg'\\
      &=
      \left|\left|f\right|\right|_1   \left|\left|f'\right|\right|_1   \left|\left| c_{h{''}, h{''}}\right|\right|_p^P<\infty.
      \end{align*}
    \end{proof}

    \medskip
    \begin{Rem}\label{convex} Let $p\in [1, \infty[$. Then the function 
function $x\mapsto x^p$  is convex for $x>0$. This means that for 
 $x_1, \ldots, x_n >0$, and $\lambda_1, \ldots,
\lambda_n\ge 0$, $\lambda_1+\cdots+\lambda_n=1$, we have
$(\lambda_1 \cdot x_1+\cdots +\lambda_n\cdot  x_n)^p\leq \lambda_1
\cdot x_1^p +\cdots +
\lambda_n \cdot x_n^p.$ Now, if $H$ is a measurable space, and $\alpha\ge 0$
is a measurable function on $H$ such that $\int_H\alpha(h)dh =1$. Then for every measurable function
$f: H\rightarrow \bbC$, we have the following inequality:
$$
\left(\int_H |f(h)|\cdot \alpha(h)dh\right)^p\leq 
\int_H |f(h)|^p \cdot \alpha(h)dh.
$$
This follows from the inequality above considering integral
sums and taking the limit. We leave details to the reader.
\end{Rem}
    
    \vskip .2in
The following lemma is one of the main technical  results of the present section:

\begin{Lem}\label{ids-2}
  Assume that $(\pi, \mathcal H)$ is integrable.  Put
  $c_{h, h'}(g)=\langle \pi(g)h, \ h'\rangle$, $h, h'\in \mathcal H_K$. 
  Then, we have the following:
  \begin{itemize}
    \item[(i)] $c_{h, h'}$ is a smooth vector in the Banach representation $L^1(G)$, where $G$ acts by right--translations.
    \item[(ii)]   Let us fix $h'\in \mathcal H_K$, $h'\neq 0$. Then, the map
      the map $c_{h'}: \ h\longmapsto c_{h, h'}$          
      is an infinitesimal  equivalence of $(\pi, \mathcal H)$ with an closed admissible
      irreducible subrepresentation $\mathcal B_{h'}$
      of $G$ in the Banach representation
      on $L^1(G)$ given by the right--translations. In particular, we have the following:
      \begin{itemize}
      \item[(a)] $c_{h'}\left(\mathcal H_K\right)$  is the space of $K$--finite vectors in $\mathcal B_{h'}$. It is contained
        in $\mathcal B_{h'}^\infty$.
      \item[(b)] We have $c_{h'}\left(\mathcal H_K(\delta)\right) = \mathcal B_{h'}^\infty(\delta)=\mathcal B_{h'}(\delta)$, for all
        $\delta\in \hat{K}$.
      \item[(c)] If  $\chi_\pi$ is the infinitesimal character of $\pi$, then all $f\in  \mathcal B^\infty_{h'}$ transforms according to
        $\chi_\pi$: $z.f=\chi_\pi(z)f$, \ \ $f\in \mathcal B^\infty_{h'}$, $z\in \mathcal Z(\mathfrak g_{\mathbb C})$. In particular,
        the vectors in  $B^\infty_{h'}$ are $\cal Z(\mathfrak g_{\mathbb C})$--finite.
        \end{itemize}
    \item[(iii)]  For each neighborhood $V$ of $1$ and $h'\in \mathcal H_K$, there exists
      $\beta\in C_c^\infty(G)$, $\supp{(\beta)}\subset V$,  such that
     $$
      l(\beta)f= f, \ \ f\in \mathcal B_{h'}.
     $$
     In particular, $\mathcal B_{h'}$ consists of smooth vectors in the left--regular representation
     of $G$ on $L^1(G)$, and  $\mathcal B_{h'}\subset C^\infty(G)$ .
   \item[(iv)]  There exists a finite set $S\subset \hat{K}$, such that
     $$
     E^l_S\left(f\right)=f,  \ \ f\in \mathcal B_{h'}.
    $$
    In particular, we have the following:
    $$
    E^l_S\left(\mathcal B_{h'}\right)=\mathcal B_{h'}.
    $$

  \item[(v)] For $u\in \mathcal U(\mathfrak g_{\mathbb C})$ and $h'\in \mathcal H_K$,
    we have $l(u)\mathcal B_{h'}\subset \mathcal B_{d\pi(u)h'}$.    

  \item[(vi)]  Let us fix $h'\in \mathcal H_K$, $h'\neq 0$. Select $h{''}\in \mathcal H$ such that $\langle h{''},
    h'\rangle=d(\pi)$, where $d(\pi)$ is the formal degree of $\pi$.  Then, the map $d_{h{''}}: \mathcal B_{h'}\longrightarrow
    \mathcal H$ defined by
    $$
    f\in \mathcal B_{h'}\longmapsto \ \pi(f^\vee)h{''} =\int_G f^\vee(x) \pi(x)h{''} dx
    $$
    is a continuous $G$--invariant embedding with the dense image. It satisfies:
    $$
    ||d_{h{''}}\left(f\right)||\le ||h{''}|| \cdot ||f||_1, \ \ f\in \mathcal B_{h'}.
    $$
    Moreover, we have the following:
    $$
    d_{h{''}}\left(c_{h'}\left(h\right)\right)=h, \ \ h\in \mathcal H_K.
    $$
    \item[(vii)] Finally, the restriction of $d_{h{''}}$ is a continuous isomorphism of Fr\' echet representations  $\mathcal B_{h'}^\infty\simeq
      \mathcal H^\infty$.    Its inverse is the map $h\in \mathcal H^\infty \longrightarrow c_{h, h'}$. In particular,
      $\mathcal B_{h'}^\infty=\left\{ c_{h, h'}; \ \ h\in \mathcal H^\infty\right\}.
      $
  \end{itemize}
  \end{Lem}
  \begin{proof} First, we prove (i). The function $c_{h, h'}$ is
    $\cal Z(\mathfrak g_{\mathbb C})$--finite  and $K$--finite on the right. Hence, by (\cite{hc}, Theorem 1), there exists
    $\alpha\in C_c^\infty(G)$ (depending of $h$) such that
    $$
    c_{h, h'}=c_{h, h'}\star \alpha= r(\alpha^\vee)c_{h, h'}.
    $$
    In view of the discussion before the statement of the theorem, this proves (i). Now, we prove (ii). 
    Clearly, the space all $c_{h, h'}$, $h\in \mathcal H_K$ is an irreducible $(\mathfrak g, K)$ isomorphic to
  $\mathcal H_K$. The closure $\mathcal B=\mathcal B_{h'}$  of such functions in $L^1(G)$ is clearly $G$--invariant. For each
  $\delta\in \hat{K}$, we have $E_\delta(\mathcal B)$ is spanned by the functions $c_{h, h'}$, $h\in \mathcal H_K(\delta)$.
  Indeed, for $f\in \mathcal B$, there exists a sequence of vectors $h_n$, $n\ge 1$ such that
  $$
  \lim_n \int_G\left| f(x)-c_{h_n, h'}(x)\right| dx =0.
  $$
  Hence, for $\delta\in \hat{K}$, we have the following:
  $$
  E_\delta(c_{h_n, h'})(x)=\int_{K} d(\delta) \overline{\xi_\delta(k)}c_{h_n, h'}(xk) dk=
  \left\langle \pi(x) E_\delta(h_n), \ h'\right\rangle=c_{E_\delta(h_n), h'}(x), \ \ x\in G.
$$
  Now, we have the following:

  \begin{align*}
    \int_G\left| E_\delta(f)(x)-c_{E_\delta(h_n), h'}(x)\right|dx & =  \int_G\left| E_\delta(f)(x)-E_\delta(c_{h_n, h'})(x)\right| dx\\
    & =  \int_G\left| \int_{K} d(\delta) \overline{\xi_\delta(k)} \left(f(xk)- c_{h_n, h'}(xk)\right)dk\right| dx\\
    &\le  \int_{K} d(\delta) \left|\overline{\xi_\delta(k)} \right| dk \int_G\left| f(x)-c_{h_n, h'}(x)\right|dx.
   \end{align*}
  This shows
  $$
  c_{E_\delta(h_n), h'} \overset{L^1}{\longrightarrow} E_\delta(f).
  $$
  But the sequence belongs to a finite--dimensional subspace of $\mathcal B$. Which is because of that closed. Hence the claim.
   In particular, $\mathcal B$ is admissible.
   Moreover, this also proves the claims in (b), and consequently in (a) since by standard argument: the vectors in
   $c_{h'}\left(\mathcal H_K\right)$ are $\cal Z(\mathfrak g_{\mathbb C})$--finite and $K$--finite. Hence, real analytic, and
   in particular, smooth. The claim in (c) is  also standard. It is true for $f\in c_{h'}\left(\mathcal H_K\right)$  which
   is the space of $K$--finite vectors in $\mathcal B_{h'}$. But this space is dense in the Fr\' echet space 
   $\mathcal B^\infty_{h'}$ (see Lemma \ref{ids-1} (ii)).  The description of topology on $\mathcal B^\infty_{h'}$  given in the second
   paragraph of this section immediately implies the claim. 
   
  Finally, we show that $\mathcal B$ is irreducible. If $\mathcal B'$ is non--zero  closed subrepresentation of
  $\mathcal B$. Then, for
  each $\delta\in \hat{K}$, we have
  $$
  \mathcal B'(\delta)=E_\delta(\mathcal B')\subset E_\delta(\mathcal B)=\mathcal B(\delta)=
  c_{h'}\left(\mathcal H_K(\delta)\right).
  $$
  But then by (i) and the irreducibility of $(\mathfrak g, K)$--module  $\mathcal H_K$, we have  
  $$
  c_{h'}\left(\mathcal H_K\right) \subset \mathcal B'_K.
  $$
  Hence, by definition of $\mathcal B$, we have $\mathcal B'=\mathcal B$.

  Let us prove (iii). We may consider $\mathcal H$ to be a subrepresentation of $L^2(G)$ under the right--translations.
  Since $h'\in \mathcal H_K$, $h'$  is a $\cal Z(\mathfrak g_{\mathbb C})$--finite  and $K$--finite on the right. So, by
  (\cite{hc}, Theorem 1), for each neighborhood $V$ of $1\in G$,  there exists $\beta_1\in C_c^\infty(G)$,
  $\supp{(\beta_1)}\subset V$ such that
  $$
  \pi(\beta_1)h'=h'\star \beta_1^\vee =h'.
  $$
  Now, we compute
  \begin{align*}
    c_{h, h'}(x) &=\int_G \langle \pi(x)h, \ h'\rangle dx \\
    &=\int_G \langle \pi(x)h, \ \pi(\beta_1)h'\rangle dx\\
    &=\int_G \int_G \overline{\beta_1(y)} \langle \pi(x)h, \ \pi(y)h'\rangle dxdy\\
    &=\int_G \overline{\beta_1(y)} \left(\int_G  \langle \pi(y^{-1}x)h, \ h'\rangle dx\right) dy\\
    &=\overline{\beta_1}\star c_{h, h'}(x). \\
    \end{align*}
  Now, let $\beta=\overline{\beta_1}$ and apply the Dominated convergence theorem to get (iii).

  The proof of (iv) is similar to the proof of (iii).  Since $h'$ is $K$--finite, there exists a finite set
  $T\subset \hat{K}$ such that $E_T(h')=h'$,  $E_T=\sum_{\delta\in \hat{K}} E_\delta$. Now, the reader can easily
  adapt the argument in (iii) to get (iv).

  We prove (v). Let $h\in\mathcal H_K$. Then, it is easy to see 
  $$
  l(u)c_{h, h'}=c_{h, \pi(u)h'}.
  $$
  Thus, we have a linear map from the space of all functions $c_{h, h'}$, $h\in\mathcal H_K$, into
  $\mathcal B_{\pi(u)h'}$. We show that it extend to a bounded map $\mathcal B_{h'}\longrightarrow \mathcal B_{\pi(u)h'}$.
  This follows from (iii) and the expression (\ref{ids-2000a}). Let us select  $\beta\in C_c^\infty(G)$ such that
  $f=\beta\star f$ for all $f\in \mathcal B_{h'}$. Thus, the map $f\longmapsto l(u)f$ can be with the aid of
   (\ref{ids-2000a}) written as follows:
$$
  f\longmapsto l(u)f=\left(l(u)\beta\right)\star f
  $$
which is clearly bounded map.

  We prove (vi).  It is well--known that $F\in L^1(G)$ acts on the unitary representation $(\pi, \mathcal H)$ as follows:
  $\pi(F)h=\int_Gf(x) \pi(g)hdx$, $h\in \mathcal H$. Moreover, we have $||\pi(F)||\le ||h|| \cdot ||F||_1$. This immediately 
  implies that the maps in (vi) is well--defined and continuous. Since
  $  \left(r(g)F\right)^\vee =l(g) F^\vee$, it is also $G$--invariant. The rest of the claims in (vi) follow from the
  Schur orthogonality relation:
  \begin{equation}\label{ids-2000}
  \int_G \langle \pi(x) h{''}, \ h_1\rangle \cdot \overline{  \langle \pi(x) h', \ h\rangle } dx=
  \frac{1}{d(\pi)}  \langle  h, \ h_1\rangle  \langle h{''}, \ h'\rangle= \langle  h, \ h_1\rangle, \ \ h_1\in \mathcal H,
  \ \ h\in \mathcal H_K.
  \end{equation}
  Since $\pi$ is unitary, the right--hand side can be transformed as follows:
  $$
  \int_G c_{h, h'}^\vee (x) \langle \pi(x) h{''}, \ h_1\rangle  dx.
  $$
  Hence, (\ref{ids-2000}) implies
  \begin{equation}\label{ids-0001}
  d_{h{''}}\left(c_{h'}\left(h\right)\right)=  \int_G c_{h, h'}^\vee (x) \ \pi(x) h{''}   dx=h, \ \ h\in \mathcal H_K.
  \end{equation}
  This identity implies that image of $d_{h{''}}$ is dense in $\mathcal H$ since it contains $K$--finite vectors.

  Finally, we prove (vii). It is obvious that the restriction of $d_{h{''}}$ is a continuous embedding of Fr\' echet spaces
  $\mathcal B_{h'}^\infty\hookrightarrow  \mathcal H^\infty$.  Using Lemma \ref{ids-0000}, we can extend the map $c_{h'}$ to $\mathcal H^\infty$
  into $L^1(G)$. Let us show that this map has image contained in $\mathcal B_{h'}^\infty$. Let $h\in \mathcal H^\infty$. Then, by Lemma \ref{ids-1} (i),
  there exists  $h_1, \ldots, h_l\in \mathcal H^\infty$, and  $\alpha_1, \ldots, \alpha_l\in C_c^\infty(G)$ such that
    $$
    h=\sum_{i=1}^l \pi(\alpha_i)h_i.
    $$
    Then, we have the following:
    $$
    c_{h, h'}(x)= \sum_{i=1}^l \int_G \alpha_i(y) \langle \pi(xy)h_i, h'\rangle dy= \sum_{i=1}^l r(\alpha_i)c_{h_i, h'}(x).
    $$
    This implies that $c_{h, h'}$ is a smooth vector for the right action of $G$. Then, by Lemma \ref{ids-1} (ii), $c_{h, h'}\in \mathcal B_{h'}$.
    Finally, since it is smooth vector, $c_{h, h'}\in \mathcal B_{h'}^\infty$. Next, being based on 
    the   Schur orthogonality relation, (\ref{ids-0001}) is true for $h\in \mathcal H^\infty$. Which shows that $d_{h'}$ is bijective continuous map between
    Fr\' echet spaces. Hence, by the Open mapping theorem,  $d_{h'}$ is an isomorphism. Since $c_{h'}$ is its inverse, we obtain (vii).
  \end{proof}

\vskip .2in 
In the next lemma we do not assume in advance that
$G$ poses discrete series. But when favorable functions exist this must be the case.  It is an analogue
of (\cite{hc}, Lemma 77) for $L^2(G)$.

\vskip .2in 

\begin{Lem}\label{ids-200} Assume  $\varphi\in C^\infty(G)\cap L^1(G)$  that is  non--zero
  $\cal Z(\mathfrak g_{\mathbb C})$--finite and 
$K$--finite on the right.  Then, under the right translations, the  $(\mathfrak g, K)$--module generated by $\varphi$
  is a  direct sum of finitely many irreducible representations each infinitesimally equivalent to an integrable discrete 
  series of $G$. In particular, $G$ poses discrete series.  
\end{Lem}
\begin{proof}   First, the standard argument  shows that
  $\varphi \in L^2(G)$. We recall that argument   (see the proof of (\cite{MuicMathAnn}, Theorem 3.10) and
  (\cite{Borel-SL(2)-book},  Corollary 2.22) ).
  Since $\varphi$ that is  $\cal Z(\mathfrak g_{\mathbb C})$ and
  $K$--finite on the right, by the result of Harish--Chandra (see \cite{hc}, Theorem 1) there exists $\alpha\in C_c^\infty(G)$ such
  that  $\varphi=\varphi\star \alpha$. Since, $\varphi\in L^1(G)$, this immediately shows that $\varphi$ is bounded. Finally,
  we write  $G_{\ge 1}$ and $G_{<1}$ for the set of all $x\in G$ satisfying $|\varphi(x)|\ge 1$ and $|\varphi(x)|<1$, respectively.
Let $C> 0$ be the  bound of $\varphi$. Then, we have the following:
$$
\infty > \int_G|\varphi(x)|dx=\int_{G_{\ge 1}} |\varphi(x)|dx+\int_{G_{<1}} |\varphi(x)|dx \ge \int_{X_{\ge 1}} dx.
$$
Thus, the set $G_{\ge 1}$ has a finite measure. Finally,
$$
\int_G|\varphi(x)|^2 dx=\int_{G_{\ge 1}} |\varphi(x)|^2 dx+\int_{G_{<1}} |\varphi(x)|^2 dx \le C^2  \int_{G_{\ge 1}} dx+
\int_{G_{<1}} |\varphi(x)| dx<\infty.
$$
Of course the same is true for $u\varphi$, where $u\in \mathcal U(\mathfrak g_{\mathbb C})$.  
So, now the $(\mathfrak g, K)$--module $V$ generated by $\varphi$ which is apriori in $L^1(G)$ belongs to $L^2(G)$ and by
(\cite{W1}, Corollary 3.4.7 and Theorem  4.2.1) it is direct sum of finitely many irreducible representations each of which
is infinitesimally isomorphic to an  discrete series:

\begin{equation}\label{ids-202}
V=V_1\oplus \cdots \oplus V_l.
\end{equation}
In particular, this forces that $rank(K)=rank(G)$ in order to have a non--zero $\varphi$.  We note that each
$V_i$ consists of functions which are $\varphi\in C^\infty(G)\cap L^1(G)$  that is  $\cal Z(\mathfrak g_{\mathbb C})$--finite and 
$K$--finite on the right simply because $V$ consists of such functions. Therefore, in what follows we may assume that
$V$ is irreducible.

We normalize Haar measure on $K$ as follows  $\int_{K}dk=1$. Then,
for an irreducible representation $\delta\in \hat{K}$,
we write $d(\delta)$ and $\xi_{\delta}$ for the degree and the character
of $\delta$, and let
$$
 E_\delta^l=\int_{K} d(\delta) \overline{\xi_{\delta}(k)}
 l(k) dk
$$
be  the projector on  $\delta$--isotypic component of the  left regular representation
$l$ on  $L^2(G)$.

Since $\varphi\in L^2(G)$, we have the following expansion which converges absolutely in $L^2(G)$:
$$
\varphi=\sum_{\delta\in \hat{K}}  E_\delta^l(\varphi).
$$
Hence there exists $\delta\in \hat{K}$ such that
\begin{equation}\label{ids-201}
E_\delta^l(\varphi)\neq 0.
\end{equation}

Obviously,  this function is  also in $ C^\infty(G)\in L^1(G)$, and is  $\cal Z(\mathfrak g_{\mathbb C})$--finite and 
$K$--finite on the right.  For any $\delta\in \hat{K}$ satisfying (\ref{ids-201}), $\psi \longmapsto E_\delta^l(\psi)$
is a $(\mathfrak g, K)$--intertwining operator between $V$ and $E_\delta^l(V)$. Since $V$ is assumed to be irreducible,
it is an isomorphism.

To complete the proof the lemma, we may assume that $\varphi$ is $K$--finite on the left. But because we  find it interesting
not to assume that $V$ is irreducible. We use  (\ref{ids-202}).

Since $\varphi$ is $K$--finite on the right, we can find a finite non--empty subset
$S\subset \hat{K}$, such that the operator $E_S^l$ defined by $\sum_{\delta\in S}  E_\delta^l$ satisfies
$E_S^l(\varphi)=\varphi$. Since $V$ is generated by $\varphi$ and  left and right actions commute with each other,
we obtain another decomposition of $V$ into subrepresentations:

$$
V=E_S^l(V)=E_S^l(V_1)\oplus \cdots \oplus E_S^l(V_l).
$$
For each $i$, $E_S^l(V_i)=\{0\}$ or isomorphic to $V_i$ since $V_i$ is irreducible.
Hence, for each $i$,  $E_S^l(V_1)\simeq V_i$ beacuse we must have the 
same number of irreducible modules in both decomposition of $V$.

We write
$$
\varphi=\sum_i \varphi_i, \ \ \varphi_i\in E_S^l(V_i).
$$
Now, since $\varphi$ generates $V$, we must have $\varphi_i\neq 0$ for all $i$. Next,
it is obvious that each $\varphi_i$ is also $\cal Z(\mathfrak g_{\mathbb C})$ and
  $K$--finite, and in $L^1(G)$.  Let us fix $i\in\{1, \ldots, l\}$.
Since $\varphi_i$ is also $K$--finite on the left, again by a result of  Harish--Chandra (see \cite{hc}, Theorem 1),
there exists  $\beta\in C_c^\infty(G)$ such that $\varphi_i=\overline{\beta}\star \varphi_i$. This can be written as
$$
\varphi_i(x)=\int_G \overline{\beta^\vee (y)}\varphi_i(yx) dx
$$
which easily implies that $\varphi_i$ is a $K$--fine matrix coefficient of $V_i$.  Indeed, we have

\begin{align*}
  \varphi_i(x) &=E_S^l(\varphi_i)(x)=\int_{K} d(\delta) \overline{\xi_{\delta}(k)}\varphi_i(k^{-1}x)dk\\
  &=\int_{K} d(\delta) \overline{\xi_{\delta}(k)}\left(\int_G \overline{\beta^\vee (y)}\varphi_i(yk^{-1}x) dx\right)dk\\
  &=\int_{K} d(\delta) \overline{\xi_{\delta}(k)}\left(\int_G \overline{\beta^\vee (yk)}\varphi_i(yx) dx\right)dk\\
  &=\int_G E_S^r\left( \overline{\beta^\vee}\right)(y)\varphi_i(yx) dx\\
  &=\int_G \overline{E_{\wit{S}}^r\left(\beta^\vee\right)(y)}\varphi_i(yx) dx,\\
\end{align*}
where $E_S^r$ is analogously defined for the right translations, and $\wit{S}=\{\wit{\delta}: \ \ \delta \in S\}$ the set
of contragredient representations. In the last equality we may replace $E_{\wit{S}}^r\left(\beta^\vee\right)$ with its orthogonal
projection to $Cl(V_i)$ and we are done.  Having proved that $\varphi_i$ is a $K$--finite matrix coefficient of $V_i$, we
immediately get each $V_i$ is integrable. 
\end{proof}

\vskip .2in
The following can be seen from the last part of the proof:

\begin{Lem}\label{ids-203} Assume  $\varphi\in C^\infty(G)\cap L^1(G)$  that is  $\cal Z(\mathfrak g_{\mathbb C})$--finite and 
  $K$--finite on the left and right. Assume that under the right translations, the  $(\mathfrak g, K)$--module generated by
  $\varphi$ is irreducible. Let $\mathcal H$ and $\mathcal B$ be the closures of $V$ in $L^2(G)$ and $L^1(G)$, respectively.
  Then,  the representation of $G$ on $\mathcal H$ by right translations is integrable, and
  there exists $h'\in \mathcal H_K$ such that (see Lemma \ref{ids-2})
  $$ 
  \mathcal B=\mathcal B_{h'}.
  $$
\end{Lem}

\vskip .2in
Finally, we prove the following theorem:

\begin{Thm}\label{ids-204} Let $\mathcal B$ the an irreducible closed admissible subrepresentation of  $L^1(G)$ under the right
  translations. Then, there exists unique up to (unitary or infinitesimal) equivalence an integrable discrete
  series $(\pi, \mathcal H)$ of $G$, $\delta\in \hat{K}$, and 
  a $K$--finite vector $h'\in \mathcal H$ such that
  $$
\text{the closure in $L^1(G)$ of} \ \  E_\delta^l\left(\mathcal B\right)=\mathcal B_{h'}.
  $$
  In particular, irreducible closed admissible subrepresentations of  $L^1(G)$ are infinitesimally equivalent to
  integrable discrete series. 
  \end{Thm}
\begin{proof} We use notation introduced in the statement and in the proof of Lemma \ref{ids-2}. The fact that
  $\mathcal B$ is admissible means that for each $\delta\in\hat{K}$,  we have 

  $$
  \mathcal B\left(\delta\right)\overset{def}{=}  E_\delta\left(\mathcal B\right)
  $$
  is finite dimensional. Let $\mathcal B^\infty$ be the space of smooth vectors in $\mathcal B$. Then, (\cite{hc}, Lemma 4) we
  have that the sum of vector spaces
  $$
  \sum_{\delta\in \hat{K}} \mathcal B\left(\delta\right)\cap \mathcal B^\infty
  $$
  is dense
  $\mathcal B$.

  Hence, there exists $\delta\in \hat{K}$ such that $\mathcal B\left(\delta\right)\cap \mathcal B^\infty\neq 0$.
  We get
  $$
  \mathcal B^\infty(\delta)\overset{def}{=} E_\delta\left(\mathcal B^\infty\right)=
  \mathcal B\left(\delta\right)\cap \mathcal B^\infty\neq 0.
  $$

  We  obtained a finite dimensional space different than $0$, invariant under  $\cal Z(\mathfrak g_{\mathbb C})$.
  Thus, $\mathcal B$ contains a non--zero $\varphi\in C^\infty(G)\cap L^1(G)$
  that is  $\cal Z(\mathfrak g_{\mathbb C})$--finite and $K$--finite on the right. As in the proof of Lemma \ref{ids-200} (see
  (\ref{ids-201})),  select $\delta\in \hat{K}$ such that $E_\delta^l(\varphi)\neq 0$. Then, using arguments in the proof
  of Lemma \ref{ids-200}, we obtain the theorem.
   \end{proof}

\section{Preparation for Application to Automorphic Forms}\label{aaf}

In this section we still assume that $G$ is a semisimple connected Lie group with finite center.
We assume that $\Gamma$ is a discrete subgroup of $G$.  Then, for $\varphi\in L^1(G)$ we can form
the Poincar\' e series $P_\Gamma(\varphi)(g):=\sum_{\gamma\in \Gamma}\varphi(\gamma\cdot
g)$.  Since 
$$
\int_{\Gamma\setminus G}|P_\Gamma(\varphi)(g)|dg\le  \int_{\Gamma\setminus 
G}\left(\sum_{\gamma\in \Gamma}|\varphi(\gamma\cdot g)|\right)dg=\int_{G}|\varphi(g)|dg<+\infty,
$$
the series $\sum_{\gamma\in \Gamma}\varphi(\gamma\cdot
g)$  converges absolutely almost everywhere and  
$P_\Gamma (\varphi)\in  L^1(\Gamma \setminus G)$.
It is obvious that the map
$$
\varphi\longmapsto P_\Gamma(\varphi)
$$
is a continuous $G$--equivariant map of Banach representations
\begin{equation}
L^1(G) \longrightarrow L^1(\Gamma \setminus G).
\end{equation}
This map is never zero. Indeed, select $\varphi\in C_c^\infty(G)$, $\varphi\neq 0$,  with a support in a sufficiently small
neighborhood $V$ of $1\in G$ which satisfies $VV^{-1}\cap \Gamma=\{1\}$. Then, we have the following:
$$
P_\Gamma(\varphi)(g):=\sum_{\gamma\in \Gamma}\varphi(\gamma\cdot g)=\varphi(g), \ \ g\in V.
$$
This proves the claim.

\vskip .2in
It is considerable harder to decide when $P_\Gamma(\varphi)\neq 0$. A sufficient condition is contained in the following lemma.

\begin{Lem}\label{aaf-1}  Let $\varphi\in L^1(G)$. Then, we have the following:
  \begin{itemize}
   \item[(i)]  Assume that there exists a compact neighborhood $C$ (i.e., an open set which closure is compact)
  in $G$ such that 
$$
\int_C |\varphi(g)|dg > \int_{G-C}
|\varphi(g)|dg\  \text{ and } \ \Gamma \cap C\cdot C^{-1}=\{1\}.
$$
Then, $P_\Gamma(\varphi)\neq 0$.
\item[(ii)]   Let $\Gamma_1\supset \Gamma_2\supset \dots$ be a sequence
  of discrete subgroups of $G$ such that   $\cap_{n\ge 1}
  \Gamma_n=\{1\}$.  Then, there exists $n_0$ depending on $\varphi$ such that
   $P_{\Gamma_n}(\varphi)\neq 0$ for $n\ge n_0$.
    \end{itemize}
\end{Lem}
\begin{proof} The  claim (i) is (\cite{MuicMathAnn}, Theorem  4-1).
  Finally, the claim (iii) is (\cite{MuicMathAnn}, Corollary 4-9).
\end{proof}

\vskip .2in
The following lemma is a variation of a standard argument:

\begin{Lem}\label{aaf-2} Let $V$ be an open neighborhood of $1\in G$ such that $VV^{-1}\cap \Gamma=\{1\}$.
  Assume that $\beta\in C_c^\infty(G)$ such that $\supp{(\beta)}\subset V$,  and $\varphi\in L^1(G)$.
  Put $\psi=\beta \star \varphi$ ($\in C^\infty(G)\in
  L^1(G)$). Then,   $P_\Gamma(\psi)$ is a bounded continuous function on $G$ in $L^1(\Gamma \backslash G)$.
  More precisely, we have the following estimate:
  $$
  \left|\left|P_\Gamma(\psi)\right|\right|_\infty\le 
  \left|\left|\beta\right|\right|_\infty \left|\left|\varphi\right|\right|_1,
  $$
\end{Lem}
\begin{proof} We have
  $$
  \psi(\gamma x)=\int_G\beta(\gamma x y^{-1}) \varphi(y)dy= \int_G\beta(\gamma x y) \varphi^\vee(y)dy,
  \ \ x\in G, \ \ \gamma \in \Gamma.
  $$
  Note that $\beta(\gamma x y)\neq 0$ implies that
  $$
  y\in x^{-1}\gamma^{-1} V.
  $$
  Since $VV^{-1}\cap \Gamma=\{1\}$, the sets
  $$
  x^{-1}\gamma^{-1} V, \ \ \gamma\in \Gamma
  $$
  are disjoint. Thus, we get
  $$
  \left|P_\Gamma(\psi)(x)\right|\le 
  \sum_{\gamma\in \Gamma} \left|\psi(\gamma x)\right|\le \left|\left|\beta\right|\right|_\infty 
  \sum_{\gamma\in \Gamma} \int_{x^{-1}\gamma^{-1} V} \left|\varphi^\vee(y)\right|dy\le
  \left|\left|\beta\right|\right|_\infty \left|\left|\varphi\right|\right|_1.
  $$
 \end{proof}

\vskip .2in

Now, we prove the main result of the present section. It contains a novel approach to convergence of Poincar\' e series.
The reader now should review Lemma \ref{ids-2}.

\begin{Thm}\label{aaf-3} Assume that $G$ admits discrete series. Let $(\pi, \mathcal H)$ is an integrable discrete series of $G$.
  Let $h'\in \mathcal H_K$, $h'\neq 0$. Let $\Gamma\subset G$ be a discrete subgroup. Then, we have the following:
  \begin{itemize}
    \item[(i)] The map $\varphi \longmapsto P_\Gamma(\varphi)$ is a continuous 
  $G$--equivariant map from  the Banach representation $\mathcal B_{h'}$ into the unitary representation
      $L^2(\Gamma \backslash G)$. The image $P_\Gamma\left(\mathcal B_{h'}\right)$ is either zero, or it is an embedding
      and its closure      is an irreducible subspace unitary  equivalent to $(\pi, \mathcal H)$.
      
 \item[(ii)]  The smooth vectors 
   $\mathcal B_{h'}^\infty$ are mapped under $P_\Gamma$ into  the subspace of $\cal Z(\mathfrak g_{\mathbb C})$--finite vectors in
   $\cal A_{umg}(\Gamma\backslash G)$.
 \item[(iii)]    Furthermore,
      we may consider, $P_\Gamma: \ \mathcal B_{h'}\longrightarrow L^1(\Gamma \backslash G)
      \subset \mathcal S\left(\Gamma\backslash G\right)'$ (see Section \ref{src}). This map is a continuous map from a  Banach space
      $\mathcal B_{h'}$ into a locally convex space $\mathcal S\left(\Gamma\backslash G\right)'$. 
   
\item[(iv)]  Let $\Gamma_1\supset \Gamma_2\supset \dots$ be a sequence
  of discrete subgroups of $G$ such that   $\cap_{n\ge 1}  \Gamma_n=\{1\}$. Then, there exists
  $n_0=n_0(h')\ge 1$ such that $P_{\Gamma_n}: \mathcal B_{h'}\longrightarrow L^2(\Gamma_n \backslash G)$ is a continuous embedding
  for $n\ge n_0$.
    \end{itemize}
   \end{Thm}
\begin{proof} Let us prove (i). Let $V$ be an open neighborhood of $1\in G$ such that $VV^{-1}\cap \Gamma=\{1\}$.
By Lemma \ref{ids-2} (iii),   there exists $\beta\in C_c^\infty(G)$,  $\supp{\ (\beta)}\subset V$, such that 
   $$
    \beta\star \varphi=  l(\beta)\varphi= \varphi, \ \ \varphi\in \mathcal B_{h'}.
    $$

    By Lemma \ref{aaf-2}, we have 
 $$
  \left|\left|P_\Gamma(\varphi)\right|\right|_\infty\le 
  \left|\left|\beta\right|\right|_\infty \left|\left|\varphi\right|\right|_1,
  $$
 Thus, we have the following:
  $$
 \int_{\Gamma\backslash G} \left|P_\Gamma(\varphi)(x)\right|^2 dx \le
 \left|\left|\beta\right|\right|_\infty \left|\left|\varphi\right|\right|_1
 \int_{\Gamma\backslash G} \left|P_\Gamma(\varphi)(x)\right| dx \le
  \left|\left|\beta\right|\right|_\infty \left|\left|\varphi\right|\right|_1^2,
 $$
  by the computation given at the beginning of this section. Hence
  $$
  \left( \int_{\Gamma\backslash G} \left|P_\Gamma(\varphi)(x)\right|^2 dx\right)^{1/2} \le
 \left|\left|\beta\right|\right|_\infty^{1/2}  \left|\left|\varphi\right|\right|_1.
 $$
 This shows that the map $\varphi\longmapsto P_\Gamma(\varphi)$ is well--defined and continuous. It is obviously
 $G$--invariant.

 Let us assume that $\mathcal B_0 \overset{def}{=}P_\Gamma\left(\mathcal B_{h'}\right)\neq 0$.
 Since $\mathcal B_{h'}$ is irreducible (see Lemma \ref{ids-2}, (ii)),
 $P_\Gamma$ must be an embedding.   Let $\mathcal H_0$ be its closure in
  $L^2(\Gamma \backslash G)$. It is obviously $G$--invariant. We show that $\mathcal H_0$ is admissible.  
 Indeed, since $\mathcal B_0$ is dense in $\mathcal H_0$ and the projector $E_\delta^r$ is continuous (see Section \ref{ids})
 for each $\delta\in \hat{K}$, we obtain
 that
  $$
 P_\Gamma\left(\mathcal B_{h'}(\delta)\right)= \mathcal B_{0}(\delta)
 $$
 is dense in
$$
 \mathcal H_{0}(\delta).
 $$
 But $\mathcal B_{h'}$ is admissible (see Lemma \ref{ids-2}, (ii)). So each $\mathcal B_{0}(\delta)$ is finite dimensional.
 Thus, closed in $\mathcal H_{0}(\delta)$.  Hence
 $$
 \mathcal H_{0}(\delta)=\mathcal B_{0}(\delta).
 $$
 This proves that $\mathcal H_{0}$ is admissible. Also,  via the map $P_\Gamma$ we obtain that
 $(\mathfrak g, K)$ modules $(\mathcal B_{0})_K$ and $(\mathcal H_{0})_K$. Next,
 by Lemma \ref{ids-2} (ii),  we obtain that $(\mathfrak g, K)$ modules $(\mathcal B_{0})_K$ are $\mathcal H_K$
 are isomorphic. Thus,  we obtain (i).

 Now, we prove (ii).  Let $\varphi\in \mathcal B_{h'}^\infty$. By Lemma \ref{ids-1} (i), there exists
 $\varphi_1, \ldots, \varphi_l\in \mathcal B^\infty$, and  $\alpha_1, \ldots, \alpha_l\in C_c^\infty(G)$ such that
 $$
 \varphi= \sum_{i=1}^l \varphi_i \star \alpha_i.
 $$
 Then, by direct computation
 $$
 P_\Gamma(\varphi)= \sum_{i=1}^l P_\Gamma(\varphi_i) \star \alpha_i.
 $$
 Next, Lemma \ref{src-5} implies that
 $$
 P_\Gamma(\varphi_i) \star \alpha_i \in \cal A_{umg}(\Gamma\backslash G),
 $$
 for all $i$. This proves  $P_\Gamma(\varphi) \in \cal A_{umg}(\Gamma\backslash G)$. To complete the proof of (ii) we need to prove
 that  $P_\Gamma(\varphi)$ is $\cal Z(\mathfrak g_{\mathbb C})$--finite. Indeed, by (i), the map  $P_\Gamma: \ \mathcal B_{h'}
 \longrightarrow L^2(\Gamma \backslash G)$ is continuous. Then, $P_\Gamma$ intertwines smooth representations
 $ \mathcal B^\infty_{h'}
 \longrightarrow \left(L^2(\Gamma \backslash G)\right)^\infty$. This implies that $P_\Gamma$ commutes with the action of
 $\mathcal U(\mathfrak g_{\mathbb C})$. Thus, since $\varphi$ is smooth and $\cal Z(\mathfrak g_{\mathbb C})$--finite,
 $P_\Gamma(\varphi)$ satisfies the same.

 Now, we prove (iii). The reader should refer to Section \ref{src} for notation and
 results used here. Let us denote by $\mathcal P_\Gamma (\varphi)$ the functional
 $f\longmapsto \int_{\Gamma\backslash G} P_\Gamma(\varphi)(x)f(x)dx$, $f\in \mathcal S\left(\Gamma\backslash G\right)$.
 The reader should now refer to the description of topology on $\mathcal S\left(\Gamma\backslash G\right)'$
 (see the description of topology after Definition \ref{src-2}).
 Let $B\subset \mathcal S\left(\Gamma\backslash G\right)$ be a bounded set. Then, 
 there exists $M_B>0$ such that
 $$
 \sup_{f\in B} \ ||f||_{1, -1}\le M_B.
 $$
 This implies
 $$
 |f(x)|\le M_B \cdot ||x||_{\Gamma\setminus G}^{-1}, \ \ f\in B, \ g\in G.
 $$
 Since $||x||_{\Gamma\setminus G}\ge 1$ for all $x\in G$, we obtain
 $$
 |f(x)|\le M_B, \ \ f\in B, \ g\in G.
 $$

 Now, we prove the continuity of the map
 \begin{align*}
   ||\mathcal P_\Gamma (\varphi)||_B &=\sup_{f\in B} \ \left|\int_{\Gamma\backslash G} P_\Gamma(\varphi)(x)f(x)dx\right|\\
   &\le \sup_{f\in B} \ \int_{\Gamma\backslash G}  \left|P_\Gamma(\varphi)(x) \right| \left|f(x)\right|dx\\
   &\le M_B   \int_{\Gamma\backslash G}  \left|P_\Gamma(\varphi)(x)\right| dx\\
   &\le M_B  \int_G\left|\varphi(x)\right| dx= M_B \ ||\varphi||_1.
 \end{align*}

 Finally, we prove (iv). We consider a $K$--finite matrix coefficient $c_{h', h'}\in L^1(G)$ defined in Lemma \ref{ids-2}.
 Then, by Lemma \ref{aaf-1} (ii), there exists  $n_0=n_0(h')\ge 1$ such that $P_{\Gamma_n}\left(c_{h', h'}\right)\neq 0$.
 Thus, for such $n$, the map  $P_{\Gamma_n}: \mathcal B_{h'}\longrightarrow L^2(\Gamma_n \backslash G)$ is non--zero.
 But, by Lemma \ref{ids-2} (ii),
 $\mathcal B_{h'}$ is a closed irreducible subrepresentation of $L^1(G)$ (under the right--translations), and $P_\Gamma$ is
 continuous on $\mathcal B_{h'}$. Hence the claim in (iv). 
\end{proof}

\vskip .2in
By Lemma \ref{aaf-1} (i) the number $n_0=n_0(h')$ can be computed as follows.  Since $c_{h', h'}\in L^1(G)$, there exists
a compact neighborhood $C$ (i.e., an open set which closure is compact)  in $G$ such that 
$$
\int_C |c_{h', h'}(g)|dg > \int_{G-C}
|c_{h', h'}(g)|dg.
$$
Since $G$ is countable at infinity such $C$ exists. Now, since $\cap_{n\ge 1}  \Gamma_n=\{1\}$, there exists $n_0=n_0(h')\ge 1$
such that $\Gamma_n\cap C\cdot  C^{-1}=\{1\}$ for $n\ge n_0$. Now, we apply Lemma \ref{aaf-1} (i) to see that
$P_{\Gamma_n}(c_{h', h'})\neq 0$. When $G$ is a group of $\mathbb R$--points of a semisimple algebraic group defined
over $\mathbb Q$ and $\Gamma_n$ is a sequence of congruence subgroups, that was indicated in (\cite{MuicMathAnn}, Theorem 6.1).
Explicit computations for $G=SL_2(\mathbb R)$ were performed in \cite{MuicJNT}.

\vskip .2in
We make the following observation:
\begin{Rem} Let $\Gamma \subset G$ be a discrete subgroup. Let $h'\in \mathcal H_K$, $h'\neq 0$. Assume that 
$P_\Gamma\left(\mathcal B_{h'}\right)=0$.  Then, by Lemma \ref{aaf-1} (i), for every 
  compact neighborhood $C$ (i.e., an open set which closure is compact) in $G$ such that
  $\Gamma \cap C\cdot C^{-1}=\{1\}$ we have 
  
$$
\int_C |\varphi(g)|dg \le  \frac12\int_{G}
|\varphi(g)|dg, \ \ \varphi\in \mathcal B_{h'}.
$$
\end{Rem}

\vskip .2in
We end this section with a comment. 
So far, we have studied explicitly constructed automorphic forms $P_\Gamma(\varphi)$ and when they are non--zero. On the other
hand, given integrable discrete series $(\pi, \mathcal H)$ of $G$, then we have the following classical observation
which is in \cite{milicic2}.

\begin{Lem}[Mili\v ci\' c] \label{aaf-4}
  Assume that $G$ admits discrete series. Let $(\pi, \mathcal H)$ be an  integrable discrete series representation of $G$.
  Then, the orthogonal complement of the algebraic sum
  $$
  \sum_{h'\in \mathcal H_K} P_\Gamma\left(\mathcal B_{h'}\right)
  $$
  in $L^2(\Gamma \backslash G)$ does not contain a $G$--invariant closed subspace equivalent to $\pi$.
  In other words, $\pi$--isotypic component in  $L^2(\Gamma \backslash G)$ is given by the closure of
  the algebraic sum $\sum_{h'\in \mathcal H_K} P_\Gamma\left(\mathcal B_{h'}\right)$.
\end{Lem}
\begin{proof} We start from the following observation.  A function $f\in L^1(G)$ acts as a bounded operator on the
  unitary representation $L^2(\Gamma \backslash G)$:
  $$
  r(f).\varphi(x)= \int_G \varphi(xy) f(y) dy=  \int_G \varphi(y) f^\vee(y^{-1}x) dy= \varphi\star f^\vee(x). 
  $$
  Assuming that $f$ is a smooth vector in the Banach representation $L^1(G)$ under the left translations, we obtain that
  $f^\vee$  is a smooth vector in the Banach representation $L^1(G)$ under the right translations. Therefore,  
  the resulting function satisfies 
  $$
  r(f).\varphi\in C^\infty(G)\cap L^2(\Gamma \backslash G).
  $$ 
  The details are left to the reader but a hint is given in the computation given after the proof of Lemma \ref{ids-1}.
  Therefore, the value  $r(f).\varphi(1)$ is well--defined. 
  
  Based on this, we have the following formula for the inner product:
  \begin{equation}\label{aaf-5}
  \langle P_\Gamma(f), \varphi \rangle=\int_{\Gamma\backslash G} \varphi(y) \  \overline{P_\Gamma(f)(y)} dy=
  \int_{G}  \varphi(y)\overline{f(y)} dy= r(\overline{f}).\varphi(1)
  \end{equation}
  since $\varphi$ is $\Gamma$--invariant on the right.
  
\medskip
  Now, we show that in the orthogonal complement in $L^2(\Gamma \backslash G)$ of the (algebraic) sum of
  subspaces $\sum_{h'\in \mathcal H_K} P_\Gamma\left(\mathcal B_{h'}\right)$ there is no closed irreducible subspace
  equivalent to $\pi$. Indeed, if $W$ is  such a subspace, then by (\ref{aaf-5}) and Lemma \ref{ids-2} (iii), we have
  $$
  r(\overline{f}).\varphi(1)= \langle P_\Gamma(\overline{f}), \varphi \rangle =0, 
\ \ \varphi\in W,  \ f\in \mathcal B_{h'}, \ \ h'\in \mathcal H_K.
  $$
  We remark that $r(\overline{f}).\varphi$ is a $K$--finite function by Lemma \ref{ids-2} (iv) and the fact
  $$
  r(k)\left(r(\overline{f})\varphi\right)= r\left(\overline{l(k)f}\right).\varphi, \ \ k\in K.
  $$
  Also, for $u\in \mathcal U(\mathfrak g_{\mathbb C})$, using Lemma \ref{ids-2} (v),  we have
  $$
  r(u)\left(r(\overline{f})\varphi\right)= r\left(\overline{l(u)f}\right).\varphi.
  $$
  Thus, by our assumption
  $$
  r(u)\left(r(\overline{f})\varphi\right)(1)=r\left(\overline{l(u)f}\right).\varphi(1)=0.
  $$
  Now, we have that $r(\overline{f})\varphi$  is real--analytic, and has all derivatives at $1$ equal to zero.
  This implies
  $$
  r(\overline{f}).\varphi =0, \ \ \varphi\in W,  \ f\in \mathcal B_{h'}, \ \ h'\in \mathcal H_K.
  $$
  But $W$ is unitary equivalent to $(\pi, \mathcal H)$. Let $\varphi \in W$ be a non--zero $K$--finite vector.
  Assume that under the fixed unitary equivalence corresponds to say $h'\in \mathcal H_K$. Then, by Schur orthogonality
  (see the proof of Lemma \ref{ids-2} (vi))
  $$
  r(\overline{c_{h', h'}}).h' \neq 0.
  $$
  (In  (\ref{ids-2000}),  we let $h=h_1=h{''}=h'$.) This contradicts above equality.
 \end{proof}

\section{Application to Automorphic Forms}\label{1aaf}

In this section we apply the results of Section \ref{aaf} to prove results for automorphic forms. We start stating hypothesis
on $G$ and $\Gamma$. In this section we assume that
 $G$ is a group of $\mathbb R$--points of a
semisimple algebraic group $\mathcal G$ defined over $\mathbb Q$. Assume that $G$ is not
compact and connected.
Let $\Gamma\subset G$ be congruence subgroup with respect to the arithmetic structure
given by the fact that $\mathcal G$ defined over $\mathbb Q$ (see \cite{BJ}).
Then,  $\Gamma$ is a discrete subgroup of $G$ and it has a finite covolume. We give the details of how to construct
congruence subgroups.

 Let $\mathbb A$ (resp., $\mathbb A_f$) be the ring of adeles (resp., finite adeles) of
$\mathbb Q$.  For each prime $p$, let 
$\mathbb Z_p$ be the maximal compact subring of $\mathbb Q_p$. Then, for almost all primes $p$, the group $\mathcal G$
is unramified over  $\mathbb Q_p$; in particular, $\mathcal G$ is a group scheme over $\mathbb Z_p$, and   $\mathcal G(\bbZ_p)$
is a hyperspecial maximal compact subgroup of  $\mathcal G(\mathbb Q_p)$ (\cite{Tits}, 3.9.1).  Let $\mathcal G(\bbA_f)$
be the restricted product of all groups $G(\mathbb Q_p)$ with respect to the groups $\mathcal G(\mathbb Z_p)$ where $p$ ranges
over all primes $p$ such that $\mathcal G$ is unramified over $\mathbb Q_p$:
$$
\mathcal G(\mathbb A_f)=\prod'_p \mathcal G(\mathbb Q_p).
$$
Note that
\begin{equation}\label{1aaf-0}
  \mathcal G(\mathbb A)= \mathcal G(\mathbb R)\times \mathcal G(\mathbb A_f).
\end{equation}
The group $\mathcal G(\mathbb Q)$ is embedded into $\mathcal G(\mathbb R)$ and
$\mathcal G(\mathbb Q_p)$. It is embedded diagonally in $\mathcal G(\mathbb A_f)$ and
in $\mathcal G(\mathbb A)$.

\vskip .2in
The congruence subgroups of $\mathcal G$ are defined as follows (see \cite{BJ}). Let $L\subset \mathcal G(\mathbb A_f)$ be an
open compact subgroup. Then, considering $\mathcal G(\mathbb Q)$ diagonally embedded in $\mathcal G(\mathbb A_f)$, we may
consider the intersection

\begin{equation}\label{1aaf-00}
\Gamma_L=L\cap \mathcal G(\mathbb Q).
\end{equation}
Now, we consider $\mathcal G(\mathbb Q)$ as subgroup of $G=\mathcal G(\mathbb R)$. One easily show that the group
$\Gamma_L$ is discrete in $G$ and it has a finite covolume. The group $\Gamma_L$ is called a congruence subgroup attached to $L$.

\vskip .2in
We introduce a family  of principal congruence groups (which depend on an embedding over $\mathbb Q$ of $\mathcal G$ into
some $SL_M$).  We fix an embedding over $\mathbb Q$
\begin{equation}\label{1aaf-1}
\mathcal G \hookrightarrow SL_M
\end{equation}
with a image Zariski closed in $SL_M$. Then there exists $N\ge 1$ such that 
\begin{equation}\label{1aaf-2}
\text{$\mathcal G$ is a group scheme over $\bbZ[1/N]$ and the embedding  (\ref{1aaf-1}) is  defined over $\bbZ[1/N]$.}
\end{equation} 
We fix such $N$.

As usual, we let $\mathcal G_\bbZ=\mathcal G(\mathbb Q)\cap SL_M(\mathbb Z)$, and $\mathcal G_{\mathbb Z_p}=
\mathcal G(\mathbb Q_p)\cap SL_M(\mathbb Z_p)$ for all prime numbers $p$. We remark that 
$\mathcal G$ is a group scheme over $\mathbb Z_p$ and the embedding  (\ref{1aaf-1}) is  defined over $\mathbb Z_p$ when
$p$ does not divide  $N$. Then $\mathcal G_{\mathbb Z_p}=\mathcal G(\mathbb Z_p)$ but
$\mathcal G$ may not be unramified over such $p$.  In general, 
$\mathcal G_{\mathbb Z_p}$ is just an open compact subgroup of $\mathcal G(\mathbb Q_p)$.

Now, we are ready to define the standard congruence subgroups with respect to the embedding (\ref{1aaf-1}).
Let $n\ge 1$. Then, we let

\begin{equation}\label{1aaf-3}
\Gamma(n) =\{x=(x_{i,j})\in \mathcal G_{\mathbb Z}: \ \ x_{i,j}\equiv \delta_{i,j} \ (mod \ n)\}.
\end{equation}

\medskip
The first result of the present section is the following theorem in which we give a construct smooth automorphic forms. The proof
contains a non--standard proof of cuspidality of Poincar\' e series (see for example \cite{MuicMathAnn}, Theorem 3-10
for the standard proof).

\begin{Thm}\label{1aaf-4} Assume that $G$ admits discrete series. Let $(\pi, \mathcal H)$ is an integrable discrete series of $G$.
  Let $h'\in \mathcal H_K$, $h'\neq 0$. Then, we have the following:
  \begin{itemize}
  \item[(i)] For each congruence subgroup $\Gamma$, the map $P_\Gamma$ maps the space of smooth vectors   $\mathcal B_{h'}^\infty$
    of $\mathcal B_{h'}$ into the space of smooth cuspidal forms $\mathcal A_{cusp}^\infty\left(\Gamma\setminus G\right)$ for $\Gamma$.
  \item[(ii)]  Assume that a family of principal congruence subgroups is defined with respect to  the embedding (\ref{1aaf-1}).
    There exists $n_0$ which depends on $\pi$ and $h'$ only, such that for
    $n\ge n_0$, the map  $P_{\Gamma(n)}$ is an embedding of  $\mathcal B_{h'}^\infty$ into
    $\mathcal A_{cusp}^\infty\left(\Gamma(n)\setminus G\right)$.
\end{itemize}
  \end{Thm}
\begin{proof} We prove (i). By Theorem \ref{aaf-3} (i), the image $P_\Gamma\left(\mathcal B_{h'}\right)$ is either zero, or it is
  an embedding and its closure is an irreducible subspace unitary  equivalent to $(\pi, \mathcal H)$. If the image is zero, then
  we clearly prove (i). But if the image is not--zero, then it is in the discrete part of $L^2(\Gamma \backslash G)$.
  Since the image is  infinitesimally equivalent to a discrete series $(\pi, \mathcal H)$, it must be contained in the
  cuspidal part of  $L^2(\Gamma \backslash G)$ by a well--known result of Wallach \cite{W0}. In particular, all functions in
  $P_\Gamma\left(\mathcal B^\infty_{h'}\right)$ are $\Gamma$--cuspidal.

  Next, By Theorem \ref{aaf-3} (ii), $P_\Gamma\left(\mathcal B^\infty_{h'}\right)$ is contained in
  the subspace of $\cal Z(\mathfrak g_{\mathbb C})$--finite vectors in
  $\cal A_{umg}(\Gamma\backslash G)$ which is the space of smooth automorphic forms
  $\mathcal A_{cusp}^\infty\left(\Gamma\setminus G\right)$ (see Lemma \ref{src-3}).  Now, (i) follows. 

  The claim (ii), follows from (i) apply  the criterion given by Lemma \ref{aaf-1} (i) to any  non--zero function
  $\varphi\in \mathcal B_{h'}^\infty$.  The details are left to the reader since they are similar to the ones in the proof
  of (\cite{MuicMathAnn}, Theorem 6-1).
  \end{proof}

\medskip
Now, we consider adelic automorphic forms. An automorphic form is a function $f: \mathcal G(\mathbb A)\longrightarrow
\mathbb C$ satisfying the following conditions (see \cite{BJ}, 4.2):
\medskip
\begin{itemize}
\item[(AA-1)] $f(\gamma x)=f(x)$, for all $\gamma\in  \mathcal G(\mathbb Q)$, $x\in \mathcal G(\mathbb A)$,
\item[(AA-2)] There exists an open compact subgroup $L\subset  \mathcal G(\mathbb A_f)$ such that $f$ is right--invariant under $L$,
\item[(AA-3)] For each $x_f\in \mathcal G(\mathbb A_f)$, the function $x_\infty\longmapsto f(x_\infty, x_f)$ satisfies the analogous
      conditions to those (A-1) and  (A-3) of Section \ref{paf}.
  \end{itemize}
We remark that for an open compact subgroup $L\subset  \mathcal G(\mathbb A_f)$, there exists a finite set $C\subset
\mathcal G(\mathbb A_f)$ such that $ \mathcal G(\mathbb A_f)= \mathcal G(\mathbb Q)\cdot  C \cdot L$ (see \cite{Borel1963}). 
We may assume that 
$C$ is the set of representatives of double cosets $\mathcal G(\mathbb Q)\backslash \mathcal G(\mathbb A_f) /L$. Then, 
by (AA-1) and (AA-2), $f$ is
completely determined by the functions $f_c$ on $G$  defined by $f_c=f|_{G\times {c}}$ for any $c\in C$. The function $f_c$
is an automorphic form for $\Gamma_{cLc^{-1}}$ (see (\ref{1aaf-00})). Next, $f$ is a cuspidal automorphic form if
\begin{equation}\label{1aaf-000}
\int_{U_{\mathcal P}(\mathbb Q)\setminus U_{\mathcal P}(\mathbb A)}f(nx)dn=0 \ \ \text{(for all
  $x\in \mathcal G(\mathbb A)$)},
\end{equation}
for all proper $\mathbb Q$--parabolic subgroups $\mathcal P$ of $\mathcal G$. Here $U_{\mathcal P}$ denotes the unipotent radical of
$\mathcal P$. It is observed in (\cite{BJ}, 4.4) that $f$ is a cuspidal automorphic form if and only if $f_c$ is
a $\Gamma_{cLc^{-1}}$--cuspidal for all $c\in C$.  This is a consequence of a simple integration formula
(see for example (\cite{MuicComp}, Lemma 2.3) for complete account):

\begin{Lem}\label{1aaf-0000}
Let $\psi: U_{\mathcal P}(\mathbb Q)\backslash 
  U_{\mathcal P}(\mathbb A)\rightarrow \mathbb C$ be a continuous function which is right--invariant under some open compact subgroup
 denoted by  $L_P\subset U_{\mathcal P}(\mathbb A_f)$. Then we
  have the following formula:
$$
\int_{U_{\mathcal P}(\mathbb Q)\backslash 
  U_{\mathcal P}(\mathbb A)}
\psi(u)du=vol_{  U_{\mathcal P}(\mathbb A_f)}(L_P)\cdot  \int_{\Gamma_{L_P} \setminus
  U_{\mathcal P}(\mathbb R)}\psi(u_\infty)du_\infty,
$$
where $\Gamma_{L_P}$ is a discrete subgroup of $  U_{\mathcal P}(\mathbb A)$ defined as
before: $\Gamma_{L_P}=  U_{\mathcal P}(\mathbb Q)\cap L_P$.
\end{Lem}

\medskip

The space of all automorphic forms we denote
by $\mathcal A(\mathcal G(\mathbb Q)\backslash \mathcal G(\mathbb A))$. It is a 
$(\mathfrak g, K)\times \mathcal G(\mathbb A_F)$--module. Its submodule is the
space of all cuspidal automorphc forms $\mathcal A_{cusp}(\mathcal G(\mathbb Q)\backslash \mathcal G(\mathbb A))$.

\medskip
The smooth automorphic forms is defined by forgetting $K$--finiteness assumption in
(AA-3). The space of all smooth automorphic forms and smooth cuspidal automorphic forms 
we denote by $\mathcal A^\infty(\mathcal G(\mathbb Q)\backslash \mathcal G(\mathbb A))$ and
$\mathcal A^\infty_{cusp}(\mathcal G(\mathbb Q)\backslash \mathcal G(\mathbb A))$, respectively. The claims analogous to those above
for relation between smooth automorphic forms on $\mathcal G(\mathbb A)$ and $G$ are easily checked to be true.

\medskip
\begin{Cor}\label{1aaf-5} Assume that $G$ admits discrete series. Let $(\pi, \mathcal H)$ is an integrable discrete series of $G$.
  Let $h'\in \mathcal H_K$, $h'\neq 0$.  For each prime number $p$, we select a function $f_p\in C_c^\infty(\mathcal G(\mathbb Q_p))$ such that for almost all $p$ where $\mathcal G$
is unramified over  $\mathbb Q_p$ we have $f_p=1_{\mathcal G(\mathbb Z_p)}$ (the characteristic function of a hyperspecial 
maximal open compact  subgroup of $\mathcal G(\mathbb Q_p)$. Let $f_\infty \in  \mathcal B_{h'}^\infty$. 
Then, the Poincar\' e series
$$
\sum_{\gamma\in \mathcal G(\mathbb Q)} \left(f_\infty \otimes'_p \ f_p \right)(\gamma x), \ \ 
x\in \mathcal G(\mathbb A),
$$
converges absolutely almost everywhere to an element of $\mathcal A^\infty_{cusp}(\mathcal G(\mathbb Q)\backslash \mathcal G(\mathbb A))$. 
\end{Cor}
\begin{proof} The finite part of the tensor product $f_{fin}=\otimes'_p \ f_p$ belongs to $C_c^\infty(\mathcal G(\mathbb A_f))$. 
Therefore, there exists an open compact subgroup $L\subset \mathcal G(\mathbb A_f)$ such that $f_{fin}$ is right--invariant under 
$L$.  As we noted above, the set, say $C$, of representatives of double cosets $\mathcal G(\mathbb Q)\backslash \mathcal G(\mathbb A_f) /L$ is finite. We fix such set. Then,  there exists a finite sequence  $\gamma_1, \cdots, \gamma_k\in \mathcal G(\mathbb Q)$, and  
$c_1, \ldots, c_k\in C$ such that the cosets $\gamma_i c_k L$ are disjoint for different indices and we have 
$$
f_{fin}=\sum_{i=1}^k  f_{fin}(\gamma_ic_i) 1_{\gamma_ic_i L}.
$$   
Next, applying the decomposition (\ref{1aaf-0}), for 
$x=(x_\infty, \delta  c l)$, where $x_\infty \in G$, $\delta \in \mathcal G(\mathbb Q)$, $c\in C$, and $l\in L$, 
we can write:
\begin{align*}
\sum_{\gamma\in \mathcal G(\mathbb Q)} \left|f_\infty \otimes f_{fin} \right|(\gamma \cdot (x_\infty, \delta  c l)) &=
\sum_{\gamma\in \mathcal G(\mathbb Q)} \left|f_\infty \otimes f_{fin} \right|(\gamma x_\infty,  \gamma \delta  c l))\\
&= \sum_{\gamma\in \mathcal G(\mathbb Q)} \left| f_\infty(\gamma x_\infty)\right|  \left|f_{fin} (\gamma \delta  c)\right|\\
&=\sum_{i=1}^k \left| f_{fin}(\gamma_ic_i) \right|\sum_{\gamma\in \mathcal G(\mathbb Q)} \left|f_\infty(\gamma x_\infty) \right|
1_{\gamma_ic_i L}(\gamma \delta  c).
\end{align*}

The last expression is equal to zero if $c\neq c_i$ for all $i$, and if $c=c_j$ (for unique $j$), then the expression reduces to 
$$
\left| f_{fin}(\gamma_j c_j) \right|\sum_{\gamma\in \mathcal G(\mathbb Q)} \left|f_\infty(\gamma x_\infty) \right|
1_{\gamma_jc_j L}(\gamma \delta  c_j). 
$$
Non--zero terms in above sum comes from the case
$$
\gamma \delta  c_j\in \gamma_jc_j L
$$
 Equivalently, 
$$
\gamma_j^{-1}\gamma\delta\in \Gamma_{c_jLc^{-1}_j}.
$$ 
Thus, after changing the summation index $\gamma$ appropriately, the sum becomes
$$
\left| f_{fin}(\gamma_j c_j) \right|\sum_{\gamma\in \gamma_j \Gamma_{c_jLc^{-1}_j}\gamma^{-1}_j }
\left|f_\infty(\gamma  \gamma_j \delta^{-1} x_\infty) \right|.
$$
The last sum converges absolutely almost everywhere by the remark at beginning of Section \ref{aaf}. The same remark could be directly applied to the adelic Poincar\' e series (see \cite{MuicMathAnn}) giving the short proof, but we need this extended  argument. Indeed,  
all above computations are still valid if we remove all absolute values. In view of Theorem \ref{1aaf-4} (i), 
there exists 

$$F_j\in \mathcal A_{cusp}^\infty\left(\Gamma_{\gamma_j \Gamma_{c_jLc^{-1}_j}\gamma^{-1}_j }\setminus G\right)$$

such that
$$
F_j(\gamma_j \delta^{-1} x_\infty)=\sum_{\gamma\in \gamma_j \Gamma_{c_jLc^{-1}_j}\gamma^{-1}_j }
f_\infty(\gamma  \gamma_j \delta^{-1} x_\infty) 
$$
is everywhere for $x_\infty \in G_\infty$.  Next,  the function 
$$
x_\infty \longmapsto F_j(\gamma_j  \delta^{-1} x_\infty)
$$
belongs to 
$$
 \mathcal A_{cusp}^\infty\left(\Gamma_{\delta \Gamma_{c_jLc^{-1}_j}\delta^{-1}}\setminus G\right)
$$ 
as it is easy to check directly from the definition of a smooth cuspidal automorphic form (see Section \ref{paf}).  This implies the corollary.
\end{proof}

\medskip
Following methods of (\cite{MuicMathAnn}, Theorems 4.1) one can developed sufficient conditions that the adelic 
Poincar\' e series from  Corollary \ref{1aaf-5} is not identically zero. For example, one could fix a prime number $p_0$ and then shrink the 
support of $f_{p_0}$. There are other possibilities. We leave it to the interested reader as an exercise from the following result.

\medskip 
Before we state the proposition, using the notation of Corollary \ref{1aaf-5},  we remark that 
we can select a compact neighborhood $C_\infty\subset G$ (i.e., compact set which is a closure of an open set) such that 
\begin{equation}\label{1aaf-6}
\int_{C_\infty} \left|f_\infty(x_\infty)\right| dx_\infty > \frac12 \int_{G} \left|f_\infty(x_\infty)\right| dx_\infty 
\end{equation}
since $f_\infty\in L^1(G)$.

 \begin{Prop}\label{1aaf-7}  We maintain the assumptions of Corollary \ref{1aaf-5}.  Let $C_\infty\subset G$
be a compact set such that (\ref{1aaf-6}) holds. Then, the  Poincar\' e series
$$
\sum_{\gamma\in \mathcal G(\mathbb Q)} \left(f_\infty \otimes'_p \ f_p \right)(\gamma x), \ \ 
x\in \mathcal G(\mathbb A)
$$
is not identically zero if
$$
\mathcal G(\mathbb Q)\cap 
 \left(\prod_p \ \supp{(f_p)}\times C_\infty \right)\cdot  \left(\prod_p \ \supp{(f_p)}\times C_\infty \right)^{-1}=\left\{1\right\}.
$$
We note that $f_p=1_{\mathcal G(\mathbb Z_p)}$ for almost all $p$. Consequently, $\supp{(f_p)}=\mathcal G(\mathbb Z_p)$.
\end{Prop}
\begin{proof} Note that $f_\infty \otimes'_p \ f_p\in L^1(\mathcal G(\mathbb A))$ . Now, we apply the general criterion
(\cite{MuicMathAnn}, Theorems 4.1)  which claims that if we can find a compact neighborhood $C\subset \mathcal G(\mathbb A)$
such that 
$$
\int_{C} \left|\left(f_\infty \otimes'_p \ f_p\right)(x)\right| dx > \frac12 \int_{G} 
\left|\left(f_\infty \otimes'_p \ f_p\right)(x)\right| dx,
$$
and 
$$
\mathcal G(\mathbb Q)\cap C\cdot C^{-1}=\left\{1\right\}.
$$
Since the integrals in above inequality are factorizable, by our assumption we can take
$$C=\prod_p \ \supp{(f_p)}\times C_\infty.
$$ 
\end{proof}

\end{document}